\documentclass[12pt,reqno]{amsart}
\usepackage{amsmath,mathtools,enumerate}
\usepackage{tikz}
\usepackage[all]{xy}
\usetikzlibrary{decorations.markings}
\usepackage{xcolor}
\usepackage{amsthm}
\usepackage{amsfonts}
\usepackage{amssymb}
\usepackage{setspace}
\usepackage{youngtab}
\usepackage{ytableau}
\usepackage{empheq}
\usepackage[bookmarks,colorlinks,breaklinks]{hyperref}  
\hypersetup{linkcolor=blue,citecolor=red,filecolor=blue,urlcolor=blue} 
\theoremstyle{plain}

\definecolor{gr40}{gray}{0.40}

\newcommand{\beqn}{\begin{eqnarray}}
\newcommand{\eeqn}{\end{eqnarray}}

\newtheorem{Theorem}{Theorem}[section]
\newtheorem{Proposition}[Theorem]{Proposition}
\newtheorem{Definition}[Theorem]{Definition}
\newtheorem{Example}[Theorem]{Example}
\newtheorem{Corollary}[Theorem]{Corollary}
\newtheorem{Lemma}[Theorem]{Lemma}
\newtheorem{Remark}[Theorem]{Remark}

\newtheorem{Algorithm}[Theorem]{Algorithm}

\newcommand{\cskew}{/\!\! /}
\newcommand{\evac}{\mathrm{evac}}
\newcommand{\rect}{\mathrm{rect}}
\newcommand{\tjdt}{\mathrm{tjdt}}
\newcommand{\jdt}{\mathrm{jdt}}

\newcommand{\comp}{\mathrm{comp}}
\newcommand{\set}{\mathrm{set}}
\newcommand{\des}{\mathrm{Des}}
\newcommand{\inv}{\mathrm{Inv}}
\newcommand{\sha}{\mathrm{sh}}
\newcommand{\std}{\mathrm{std}}
\newcommand{\ctau}{\tau}
\newcommand{\ctaup}{\bar{\tau}} 
\newcommand{\ncsa}{\mathbf{s}_{\alpha}}
\newcommand{\ncsb}{\mathbf{s}_{\beta}}
\newcommand{\ncsg}{\mathbf{s}_{\gamma}}
\newcommand{\nch}{\mathbf{h}}
\newcommand{\rib}{\mathbf{r}} 
\newcommand{\ncs}{\mathbf{s}} 
\newcommand{\up}{u} 
\newcommand{\down}{\mathfrak{d}} 
\newcommand{\downrow}{\mathfrak{v}}

\definecolor{myblue}{rgb}{0.4, 0.5, 0.7}

\newcommand{\tc}{\textcolor}
\newcommand{\tcr}{\tc{red}}
\newcommand{\tcb}{\tc{blue}}
\definecolor{myblue}{rgb}{0.6, 0.7, 1}
\definecolor{anotherblue}{rgb}{0.3,0.3,1}
\definecolor{mygreen}{rgb}{0.3,0.5,0}
\definecolor{myred}{rgb}{0.8,0.3,0.1}
\newcommand*\mybluebox[1]{%
    \colorbox{myblue}{\hspace{1em}#1\hspace{1em}}}
    
\newcommand{\SRT}{\ensuremath{\operatorname{SRT}}}
\newcommand{\SSRT}{\ensuremath{\operatorname{SSRT}}}
\newcommand{\SRCT}{\ensuremath{\operatorname{SRCT}}}
\newcommand{\SSRCT}{\ensuremath{\operatorname{SSRCT}}}

\begin{document}
\title[Backward jdt slides and right Pieri rule]{Backward jeu de taquin slides for composition tableaux and a noncommutative Pieri rule}
\author{VASU V. TEWARI}
\address{
Vasu V. Tewari\\
Department of Mathematics\\
University of British Columbia\\
Vancouver, BC V6T 1Z2\\
Canada
}
\email{\href{mailto:vasu@math.ubc.ca}{vasu@math.ubc.ca}}

\begin{abstract}
We give a backward jeu de taquin slide analogue on semistandard reverse composition tableaux. These tableaux were first studied by Haglund, Luoto, Mason and van Willigenburg when defining quasisymmetric Schur functions. Our algorithm for performing backward jeu de taquin slides on semistandard reverse composition tableaux results in a natural operator on compositions that we call the jdt operator. This operator in turn gives rise to a new poset structure on compositions whose maximal chains we enumerate. As an application, we also give a noncommutative Pieri rule for noncommutative Schur functions that uses the jdt operators.
\end{abstract}
\maketitle
\tableofcontents
\section{Introduction}
In this article we introduce an analogue of a combinatorial procedure, called jeu de taquin, for semistandard reverse composition tableaux and study the implications in the setting of noncommutative symmetric functions. Classically, jeu de taquin, literally the `teasing game', is a set of sliding rules defined on semistandard Young tableaux (or semistandard reverse tableaux) that preserves the property of being semistandard while changing the underlying shape. This procedure is of utmost importance in algebraic combinatorics, representation theory and algebraic geometry. It was introduced by Sch\"{u}tzenberger \cite{Schutzenberger}, and its many remarkable properties have been studied in depth since \cite{fulton-1, sagan, stanley-ec2}. The procedure plays a crucial role in the combinatorics of permutations and Young tableaux, with deep links to the Robinson-Schensted algorithm and the plactic monoid. One of the most important applications of Sch\"{u}tzenberger's jeu de taquin is in giving a combinatorial rule for computing the tensor product of two representations of the symmetric group, commonly called the Littlewood-Richardson rule.

Analogues of jeu de taquin have been found for shifted tableaux \cite{Haiman,Sagan-2}, Littelmann's crystal paths \cite{vanLeeuwen}, increasing tableaux \cite{TY-1}, edge labeled Young tableaux \cite{TY-2} and d-complete posets \cite{Proctor,RN}. Procedures such as promotion and evacuation, that are consequences of jeu de taquin and important in their own right, have also been studied \cite{PonWang,Stanley-1}. Recently, an infinite version of jeu de taquin has been studied for infinite Young tableaux \cite{sniady}.

Our focus here is on semistandard reverse composition tableaux. These tableaux arise naturally in the work of Haglund, Luoto, Mason and van Willigenburg \cite{HLMvW} for defining a distinguished basis of the Hopf algebra of quasisymmetric functions, called the basis of quasisymmetric Schur functions. These functions are indexed by compositions and are obtained from a certain specialization of nonsymmetric Macdonald polynomials \cite{HHL}. The dual basis elements corresponding to them, which belong to the Hopf algebra of noncommutative symmetric functions, are noncommutative Schur functions introduced in \cite{BLvW}. These functions are noncommutative lifts of the classical Schur functions, and resemble them by possessing properties that are noncommutative analogues of Pieri rules, Kostka numbers, and Littlewood-Richardson rule. 

A special case of the noncommutative Littlewood-Richardson rule proved in \cite{BLvW} is a noncommutative analogue of the classical Pieri rules. These correspond to expanding the product of noncommutative Schur functions $\ncs_{(n)}\cdot \ncs_{\alpha}$ and $\ncs_{(1,\ldots,1)}\cdot \ncsa$ in the basis of noncommutative Schur functions, where $n$ denotes a positive integer and $\alpha$ denotes a composition. Given that we are computing products in a noncommutative Hopf algebra, it is natural to consider the products $\ncs_{\alpha}\cdot \ncs_{(n)}$ and $\ncs_{\alpha}\cdot \ncs_{(1,\ldots,1)}$. The statement of the noncommutative Littlewood-Richardson rule in \cite{BLvW}, which allows us to compute the aforementioned products, requires rectification of standard reverse tableaux and a map defined by Mason in \cite{Mason-rho}, denoted by $\rho$, that links semistandard reverse tableaux and semistandard reverse composition tableaux. One could ask whether it is possible to eliminate the use of Mason's map and introduce an analogue of rectification for semistandard reverse composition tableaux. This was achieved by Bechard in \cite{Bechard}. For what we aspire to compute, it is preferable to consider a procedure that is the inverse of Bechard's rectification.

The outline of this article is as follows. After reminding the reader of some classical objects from the theory of symmetric functions in Section \ref{section:background sym}, we proceed to describe the noncommutative analogues of the same in the algebra of noncommutative symmetric functions, in Section \ref{section: background nsym}. In the same section, we introduce the distinguished basis of noncommutative Schur functions after introducing semistandard reverse composition tableaux. We introduce certain operators on compositions, namely box-removing operators and jdt operators, in Section \ref{section: operators on compositions}. We state our two main results in Section \ref{section: Results}. First, the procedure for executing the backward jeu de taquin slides on semistandard reverse composition tableaux is outlined using Algorithms \ref{results algorithm1} and \ref{results algorithm2}. Second, a multiplicity-free expansion for the products $\ncs_{\alpha}\cdot \ncs_{(n)}$ and $\ncs_{\alpha}\cdot \ncs_{(1,\ldots,1)}$ in terms of the noncommutative Schur functions, namely the right Pieri rule, is stated in Theorem \ref{theorem: right pieri rules}. We then proceed to give the proofs of validity of the algorithms described earlier in Section \ref{section: backward jdt slides straight shape} and the connection is made precise in Theorem \ref{backwardjdtslideonssrct}. Next, in Subsection \ref{section: jdt operators}, we study the effect of the slides on underlying composition shapes using jdt operators. In Section \ref{section: slides for skew shape}, we describe a procedure to perform backward jeu de taquin slides on skew semistandard reverse composition tableaux. Using the jdt operators, in Section \ref{section: new poset} we endow the set of compositions with a new poset structure $\mathcal{R}_c$. Finally, in Section \ref{section: right pieri rule}, we prove the right Pieri rule in Theorem \ref{theorem: right pieri rules actual proof} and enumerate the maximal chains in $\mathcal{R}_c$ in Theorem \ref{theorem: generalized counting maximal chains in right pieri poset}.

\section*{Acknowledgements}
The author would like to thank Stephanie van Willigenburg for helpful guidance. He also wishes to thank Melissa Bechard, Matthieu Josuat-Verg\`{e}s, Sarah Mason, and Ed Richmond for helpful conversations.
\section{Backgroung on symmetric functions}\label{section:background sym}
We will start by defining some combinatorial structures that we will encounter frequently. All the notions introduced in this section are covered in detail in \cite{macdonald-1,sagan,stanley-ec2}. 
\subsection{Partitions}
A \emph{partition} $\lambda$ is a finite list of positive integers $(\lambda_1,\ldots,\lambda_k)$ satisfying $\lambda_1\geq \lambda_2\geq \cdots\geq \lambda_k$. The integers appearing in the list are called the \emph{parts} of the partition.

Given a partition $\lambda=(\lambda_1,\ldots,\lambda_k)$, the \textit{size} $\lvert \lambda\rvert$ is defined to be $\sum_{i=1}^{k}\lambda_i$. The number of parts of $\lambda$ is called its \textit{length}, and is denoted by $l(\lambda)$. If $\lambda$ is a partition satisfying $\lvert \lambda\rvert=n$, then we denote this by $\lambda \vdash n$. By convention, there is a unique partition of size and length equalling $0$, and we denote it by $\varnothing$.

We will depict a partition using its \textit{Young diagram}. Given a partition $\lambda=(\lambda_1,\ldots,\lambda_k)\vdash n$, the Young diagram of $\lambda$, also denoted by $\lambda$, is the left-justified array of $n$ boxes, with $\lambda_i$ boxes in the $i$-th row. We will be using the French convention where the rows are numbered from bottom to top and the columns from left to right. We refer to the box in the $i$-th row and $j$-th column by the ordered pair $(i,j)$. If $\lambda$ and $\mu$ are partitions such that $\mu \subseteq \lambda$, that is, $l(\mu) \leq l(\lambda)$ and $\mu_{i} \leq \lambda_{i}$ for all $i=1,2,\ldots,l(\mu)$, then the \textit{skew shape} $\lambda / \mu$ is obtained by removing the first $\mu_{i}$ boxes from the $i$-th row of the Young diagram of $\lambda$ for $1\leq i\leq l(\mu)$. Given a skew shape $\lambda/\mu$, we call $\mu$ the \textit{inner shape} and $\lambda$ the \textit{outer shape}. We refer to $\lambda$ as the \textit{outer shape} and to $\mu$ as the \textit{inner shape}. If the inner shape is $\varnothing$, instead of writing $\lambda/\varnothing$, we just write $\lambda$ and call $\lambda$ a \textit{straight shape}. The \emph{size} of a skew shape $\lambda/ \mu$, denoted by $\vert \lambda/\mu\vert$, is the number of boxes in the skew shape, which is $\vert \lambda\vert-\vert\mu\vert$.

Next, we define addable nodes and removable nodes of a partition. An \textit{addable node} of a partition $\lambda$ is a position $(i,j)$ where a box can be appended to a part of $\lambda$ so that the resulting shape is still a partition. Note that the addable nodes of a partition do not belong to the partition. On the other hand, a \textit{removable node} of $\lambda$ is a position $(i,j)$ where a box can be removed from a part of $\lambda$ so that the resulting shape is still a partition. Unlike addable nodes, removable nodes of a partition belong to the partition. Furthermore, given a partition, both addable nodes and removable nodes are uniquely determined by the column to which they belong. In the definitions just given, we have identified the partition $\lambda$ with its Young diagram, also called $\lambda$.

We will define a lattice structure on the set of partitions, called \textit{Young's lattice}, as follows. If $\mu$ is a partition obtained by adding a box at an addable node of the partition $\lambda$, we say that $\mu$ covers $\lambda$ and we denote it by $\lambda \prec \mu$. The partial order obtained by taking the transitive closure of $\prec$ allows to endow the set of partitions with the structure of a lattice, that we will denote $\mathcal{Y}$.
\begin{Example}
The partition $\lambda = (3,2,2,1)$ is represented by the following Ferrers diagram.
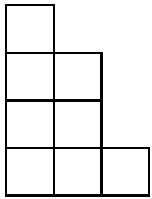
\begin{eqnarray*}
\ydiagram{1,2,2,3}
\end{eqnarray*}
The box in the lower left corner is indexed by the pair $(1,1)$. The addable nodes of $\lambda$ are $(1,4)$, $(2,3)$, $(4,2)$ and $(5,1)$ while the removable nodes are $(1,3)$, $(3,2)$ and $(4,1)$.
Also, the partition $\mu=(3,3,2,1)$ covers $\lambda$ in $\mathcal{Y}$, that is, $(3,2,2,1)\prec (3,3,2,1)$.
\end{Example}

\subsection{Semistandard reverse tableaux}
In this subsection, we will introduce certain classical objects that play a central role in the theory of symmetric functions.
\begin{Definition}\label{def:SSRT}
Given a skew shape $\lambda/\mu$, a \emph{semistandard reverse tableau ($\SSRT$)} $T$ of \emph{shape} $\lambda/\mu$ is a filling of the boxes of $\lambda/\mu$ with positive integers, satisfying the condition that the entries in $T$ are weakly decreasing along each row read from left to right and strictly decreasing along each column read from bottom to top.
\end{Definition}
We will denote the set of all $\SSRT$s of shape $\lambda/\mu$ by $\SSRT(\lambda/\mu)$. Given an $\SSRT$ $T$, the entry in box $(i,j)$ is denoted by $T_{(i,j)}$. The shape underlying $T$ is denoted by $\sha(T)$.

A \textit{standard reverse tableau} ($\SRT$) $T$ of shape $\lambda/\mu$ is an $\SSRT$ that contains every positive integer in $[|\lambda/\mu|]=\{1,2,\ldots,|\lambda/\mu|\}$ exactly once. 

The \textit{column reading word} of an $\SSRT$ is the word obtained by reading the entries of every column in increasing order from left to right.

Let $\{x_1,x_2,\ldots\}$ be an alphabet comprising of countably many commuting indeterminates. Given any $\SSRT$ $T$ of shape $\lambda\vdash n$, we will associate a monomial $x^{T}$ with it as follows.
\begin{eqnarray*}
x^{T}=\prod_{(i,j)\in\lambda}x_{T_{(i,j)}}
\end{eqnarray*}
\begin{Example}
Shown below are an $\SSRT$ $T$ of shape $(4,3,3,1)$ and its associated monomial.
\begin{eqnarray*}
T=\ytableausetup{mathmode,boxsize=1em}
\begin{ytableau}
1\\3&3&2\\6&5&4\\7&7&6&3
\end{ytableau} \hspace{6mm} 
x^{T}=x_7^2x_6^2x_5x_4x_3^3x_2x_1
\end{eqnarray*}
The column reading word of $T$ is $1367 \text{ } 357 \text{ }246\text{ }3$.
\end{Example}

\subsection{Symmetric functions}
The algebra of symmetric functions, denoted by $\Lambda$, is the algebra freely generated over $\mathbb{Q}$ by countably many commuting variables $\{h_1,h_2,\ldots\}$. Assigning degree $i$ to $h_i$ (and then extending this multiplicatively) allows us to endow $\Lambda$ with the structure of a graded algebra. A basis for the degree $n$ component of $\Lambda$, denoted by $\Lambda^{n}$, is given by the \textit{complete homogeneous symmetric functions} of degree $n$, $$\{h_{\lambda}=h_{\lambda_1}\cdots h_{\lambda_k}: \lambda=(\lambda_1,\ldots,\lambda_k)\vdash  n\}.$$

A concrete realization of $\Lambda$ is obtained by embedding $\Lambda=\mathbb{Q}[h_1,h_2,\ldots]$ in $\mathbb{Q}[[x_1,x_2,\ldots]]$, that is, the ring of formal power series in countably many commuting indeterminates $\{x_1,x_2,\ldots\}$, under the identification (extended multiplicatively)
$$h_i \longmapsto \text{ sum of all distinct monomials in $x_1,x_2,\ldots$ of degree $i$}.$$
This viewpoint allows us to consider symmetric functions as formal power series $f$ in the $x$ variables with the property that $f(x_{\pi(1)},x_{\pi(2)},\ldots)=f(x_{1},x_{2},\ldots)$ for every permutation $\pi$ of the positive integers $\mathbb{N}$.

The \textit{Schur function} indexed by the partition $\lambda$, $s_{\lambda}$, is defined as follows.
\begin{eqnarray*}
s_{\lambda}=\displaystyle\sum_{T\in \SSRT(\lambda)}x^{T}
\end{eqnarray*}
While not immediate from the definition above, it is a fact that $s_{\lambda}$ is a symmetric function. Additionally, the elements of the set $\{s_{\lambda}:\lambda\vdash n\}$ form a basis of $\Lambda ^{n}$ for any positive integer $n$.

\subsection{Jeu de taquin slides on SSRTs}
In this subsection, we will describe backward jeu de taquin slides in the setting of semistandard reverse tableaux. The exposition here follows that in \cite{sagan} analogously.

Let $\lambda \vdash n$ and $T\in SSRT(\lambda)$. Given $i_0,j_0\geq 1$ such that $\lambda$ has an addable node in position $c=(i_0,j_0)$, a \textit{backward jeu de taquin slide} on $T$ starting from $c$ gives an $\SSRT$ $\jdt_{j_0} (T)$ in the manner outlined below.
\begin{enumerate}
\item Let $c=(i_0,j_0)$.
\item \textbf{While} $c$ is not equal to $(1,1)$ \textbf{do}\begin{itemize}\item If $c=(i,j)$, then let $c'$ be the box of $min(T_{(i-1,j)},T_{(i,j-1)})$. If only one of $T_{(i-1,j)}$ and $T_{(i,j-1)}$ exists then the minimum is taken to be that single value. Furthermore, if $T_{(i-1,j)}=T_{(i,j-1)}$, then $c'=(i-1,j)$.
\item Slide $T_{c'}$ into position $c$. Let $c:=c'$.
\end{itemize}
\item The final skew $\SSRT$ obtained is $\jdt_{j_0}(T)$.
\end{enumerate}

We give an example demonstrating the above algorithm.
\begin{Example}
Let $\lambda = (3,2,2,1)$ and let $T\in SSRT(\lambda)$ be
$\ytableausetup{mathmode,boxsize=1em}
\begin{ytableau}
2\\5 & 4\\ 6 & 5\\8 & 7 &1
\end{ytableau}$.

If we start a backward jeu de taquin slide from the addable node given by $c=(2,3)$, then the algorithm executes in the following manner.
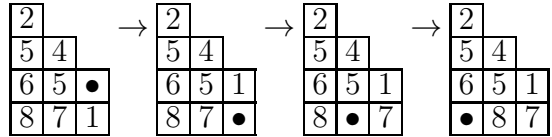
\begin{eqnarray*}
\ytableausetup{mathmode,boxsize=1em}
\begin{ytableau}
2\\5 & 4\\ 6 & 5 & \bullet \\8 & 7 &1
\end{ytableau}
\rightarrow
\begin{ytableau}
2\\5 & 4\\ 6 & 5 & 1 \\8 & 7 & \bullet
\end{ytableau}
\rightarrow
\begin{ytableau}
2\\5 & 4\\ 6 & 5 & 1 \\8 & \bullet &7 
\end{ytableau}
\rightarrow
\begin{ytableau}
2\\5 & 4\\ 6 & 5 & 1 \\ \bullet  & 8 &7 
\end{ytableau}
\end{eqnarray*}
Thus, we have that
\begin{eqnarray*}
\emph{jdt}_{3}(T)&=&\ytableausetup{mathmode,boxsize=1em}\begin{ytableau}
2\\5 & 4\\ 6 & 5 & 1 \\ \none & 8 &7 
\end{ytableau}.
\end{eqnarray*}
\end{Example}

The above algorithm can be generalized naturally to $\SSRT$ of skew shape. Now let $T$ denote an $\SSRT$ of skew shape $\lambda /\mu$. Suppose $i_0,j_0\geq 1$ are such that the position $c=(i_0,j_0)$ is an addable node of $\lambda$. Then $\jdt _{j_0}(T)$ is obtained as outlined below.

\begin{enumerate}
\item Let $c=(i_0,j_0)$.
\item \textbf{While} $c$ is not an addable node of $\mu$ \textbf{do}\begin{itemize}\item If $c=(i,j)$, then let $c'$ be the box of $min(T_{(i-1,j)},T_{(i,j-1)})$. If only one of $T_{(i-1,j)}$ and $T_{(i,j-1)}$ exists then the minimum is taken to be that single value. Furthermore, if $T_{(i-1,j)}=T_{(i,j-1)}$, then $c'=(i-1,j)$.
\item Slide $T_{c'}$ into position $c$. Let $c:=c'$.
\end{itemize}
\item The final skew $\SSRT$ obtained is $\jdt_{j_0}(T)$.
\end{enumerate}

\begin{Example}\label{example: backward jdt slide on skew SSRT}
Consider the $\SSRT$ $T$ of shape $(5,4,3,1)/ (3,2)$ as shown below.
\begin{eqnarray*}
\ytableausetup{mathmode,boxsize=1em}
\begin{ytableau}
1\\3&1&1\\ \bullet & \bullet & 2 &1\\ \bullet & \bullet & \bullet &4 &3
\end{ytableau}
\end{eqnarray*}
Suppose we want to compute $\jdt_{4}(T)$. Then the following sequence of sliding moves occurs, where the box shaded red denotes the box into which entries slide.
\begin{eqnarray*}
\ytableausetup{mathmode,boxsize=1em}
\begin{ytableau}
1\\3&1&1 & *(red)\bullet\\ \bullet & \bullet & 2 &1\\ \bullet & \bullet & \bullet &4 &3
\end{ytableau}
\longrightarrow
\begin{ytableau}
1\\3&1&1 & 1\\ \bullet & \bullet & 2 & *(red)\bullet \\ \bullet & \bullet & \bullet &4 &3
\end{ytableau}
\longrightarrow
\begin{ytableau}
1\\3&1&1 & 1\\ \bullet & \bullet & *(red)\bullet & 2 \\ \bullet & \bullet & \bullet &4 &3
\end{ytableau}
\end{eqnarray*}
Thus, we have
\begin{eqnarray*}
\emph{jdt}_{4}(T)&=&\ytableausetup{mathmode,boxsize=1em}\begin{ytableau}
1\\3&1&1 & 1\\ \none & \none & \none & 2 \\ \none & \none & \none &4 &3
\end{ytableau}.
\end{eqnarray*}
\end{Example}

For completeness, we will discuss the \textit{forward jeu de taquin} slide next. It is the inverse of the backward jeu de taquin slide discussed above. Let $T$ denote an $\SSRT$ of skew shape $\lambda /\mu$. Suppose $i_0,j_0\geq 1$ are such that the position $c=(i_0,j_0)$ is a removable node of $\mu$. Then the $\SSRT$ obtained after a forward jeu de taquin slide is initiated from position $c$, which we denote by $\jdt^{f} _{j_0}(T)$, is obtained as outlined below.

\begin{enumerate}
\item Let $c=(i_0,j_0)$.
\item \textbf{While} $c$ is not an addable node of $\lambda$ \textbf{do}\begin{itemize}\item If $c=(i,j)$, then let $c'$ be the box of $max(T_{(i+1,j)},T_{(i,j+1)})$. If only one of $T_{(i+1,j)}$ and $T_{(i,j+1)}$ exists then the maximum is taken to be that single value. Furthermore, if $T_{(i+1,j)}=T_{(i,j+1)}$, then $c'=(i+1,j)$.
\item Slide $T_{c'}$ into position $c$. Let $c:=c'$.
\end{itemize}
\item The final $\SSRT$ obtained is $\jdt^{f}_{j_0}(T)$.
\end{enumerate}
\begin{Example}
Let $T$ denote the $\SSRT$ below.
\begin{eqnarray*}
\ytableausetup{mathmode,boxsize=1em}\begin{ytableau}
1\\3&1&1 & 1\\ \bullet & \bullet & \bullet & 2 \\ \bullet & \bullet & \bullet &4 &3
\end{ytableau}
\end{eqnarray*}
Then $\jdt^{f}_{3}(T)$ is obtained by reversing the steps in Example \ref{example: backward jdt slide on skew SSRT} and we get that
\begin{eqnarray*}
\ytableausetup{mathmode,boxsize=1em}
\jdt^{f}_{3}(T)=\begin{ytableau}
1\\3&1&1\\ \bullet & \bullet & 2 &1\\ \bullet & \bullet & \bullet &4 &3
\end{ytableau}.
\end{eqnarray*}
\end{Example}

\subsection{Rectification of SRTs}\label{subsection: vrsk}
A procedure that is closely tied to the jeu de taquin slide is rectification. To describe it in a way that is convenient for us, we need to introduce a slight variant of the Robinson- Schensted algorithm. 

A discussion of the original algorithm is present in the appendix. Our variant establishes a bijection between permutations in $\mathfrak{S}_n$ and pairs of $\SRT$s of the same shape. We will need the notion of a \textit{partial tableau} that we define to be an $\SSRT$ with distinct entries. Furthemore, we will abuse notation and use $\varnothing$ to denote the empty tableau.

We will outline the algorithm next. Our input is a permutation $\sigma\in \mathfrak{S}_n$, and the output is a pair of $\SRT$s of the same shape and size $n$.
\begin{Algorithm}[Variant of RS insertion algorithm]\label{algorithm: variant RS algorithm}
\hfill
\begin{enumerate}
\item Let $P_{0}=Q_{0}=\varnothing$.
\item \textbf{For} $i=1$ to $n$, let $y\coloneqq \sigma(i)$ and \textbf{do}
\begin{enumerate}[(a)]
\item Set $r\coloneqq$ first row of $P_{i-1}$.
\item \textbf{While} $y$ is greater than some number in row $r$, \textbf{do}
\begin{itemize}
\item Let $x$ be the greatest number smaller than $y$ and replace $x$ by $y$ in $P_{i-1}$.
\item Set $y\coloneqq x$ and $r\coloneqq$ the next row.
\end{itemize}
\item Now that $y$ is smaller than every number in row $r$, place $y$ at the end of row $r$ and call the resulting partial tableau $P_{i}$. If $(a,b)$ is where $y$ got placed, the partial tableau $Q_{i}$ is obtained by placing $n-i+1$ in position $(a,b)$ in $Q_{i-1}$. 
\end{enumerate}
\item Finally, set $P\coloneqq P_{n}$ and $Q\coloneqq Q_{n}$.
\end{enumerate}
\end{Algorithm}
The $\SRT$s $P$ and $Q$ thus obtained are called the \textit{insertion tableau} and \textit{recording tableau} respectively. We present an example next.
\begin{Example}\label{example: insertion variant RSK}
Consider the permutation $\sigma = 3157624$ in one line notation. Then the insertion algorithm works as shown below, where at each stage we note the pair $(P_i,Q_i)$.
\begin{align*}
&(\varnothing,\varnothing)\mapsto
\left(\ytableausetup{mathmode,boxsize=1em}
\begin{ytableau}
3
\end{ytableau}
,
\begin{ytableau}
7
\end{ytableau}\right)\mapsto
\left(\ytableausetup{mathmode,boxsize=1em}
\begin{ytableau}
3 & 1
\end{ytableau}
,
\begin{ytableau}
7&6
\end{ytableau}\right)\mapsto
\left(\ytableausetup{mathmode,boxsize=1em}
\begin{ytableau}
3\\5 & 1
\end{ytableau}
,
\begin{ytableau}
5\\7&6
\end{ytableau}\right)\mapsto
\left(\ytableausetup{mathmode,boxsize=1em}
\begin{ytableau}
3\\5\\7 & 1
\end{ytableau}
,
\begin{ytableau}
4\\5\\7&6
\end{ytableau}\right)\mapsto
\left(\ytableausetup{mathmode,boxsize=1em}
\begin{ytableau}
3\\5 & 1\\7 & 6
\end{ytableau}
,
\begin{ytableau}
4\\5&3\\7&6
\end{ytableau}\right)\\&\mapsto
\left(\ytableausetup{mathmode,boxsize=1em}
\begin{ytableau}
3\\5 & 1\\7 & 6 & 2
\end{ytableau}
,
\begin{ytableau}
4\\5&3\\7&6&2
\end{ytableau}\right)\mapsto 
\left(\ytableausetup{mathmode,boxsize=1em}
\begin{ytableau}
3 &1\\5 & 2\\7 & 6 & 4
\end{ytableau}
,
\begin{ytableau}
4&1\\5&3\\7&6&2
\end{ytableau}\right)
\end{align*}
\end{Example} 

Now, given an $\SRT$ $T$, let $w_T$ denote its column reading word. Note that $w_T$ is also a permutation written in single line notation. The \textit{rectification} of $T$, denoted $\rect(T)$ is defined to be $P(w_T)$, that is, the insertion tableau obtained on performing the above mentioned variant of Robinson-Schensted algorithm on $w_T$.
\begin{Example}
Consider the $\SRT$ $T$ of skew shape given below.
$$
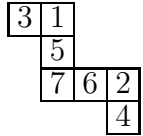
\ytableausetup{mathmode,boxsize=1em}
\begin{ytableau}
3&1\\ \none & 5\\ \none & 7&6 &2\\ \none & \none &\none & 4
\end{ytableau}
$$
Then the column reading word is $3 \text{ } 157 \text{ } 6\text{ } 24$. The insertion tableau for this reading word has already been computed in Example \ref{example: insertion variant RSK}, and thus $\rect(T)$ is as follows.
$$
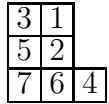
\ytableausetup{mathmode, boxsize=1em}
\begin{ytableau}
3 &1\\5 & 2\\7 & 6 & 4
\end{ytableau}
$$
\end{Example}
An important property of rectification that will prove crucial for us is the fact that it behaves nicely with respect to jeu de taquin slides. More precisely, if $T'$ is some $\SRT$ obtained by performing a backward jeu de taquin slide on some $\SRT$ $T$, then $\rect(T)=\rect(T')$. Rectification also allows us to describe another important combinatorial procedure called evacuation which we describe next.

Given an $\SRT$ $T$, the \textit{evacuation} action on $T$, denoted by $e(T)$, is obtained using the algorithm below.
\begin{Algorithm}[Evacuation]\label{algorithm:evacuation}
\hfill
\begin{enumerate}
\item Let $U$ be a Ferrers diagram of shape $\lambda$, where $\lambda = sh(T) \vdash n$.
\item Let the entry in the box in position $(1,1)$ in $T$ be $a$. Remove this entry, and compute the rectification of the remaining $\SRT$ of skew shape. Equivalently, one can remove the entry in the box in position $(1,1)$ in $T$ and let $T'$ be the resulting skew $\SRT$, and compute $\jdt^{f}_{1}(T')$. This viewpoint makes it clear that the underlying shape of the rectified tableau is obtained by removing a removable node from $\lambda$. Let this node be in position $(i,j)$.
\item To the box given by $(i,j)$ in $U$, assign the entry $n-a+1$.
\item Repeat the previous two steps until all the boxes in $U$ are filled. Finally, $e(T)=U$.
\end{enumerate}
\end{Algorithm}
\begin{Example}
Consider the insertion tableau from Example \ref{example: insertion variant RSK}. We demonstrate Algorithm \ref{algorithm:evacuation} by drawing the tableau $T$ and the partially filled $U$ at each step.
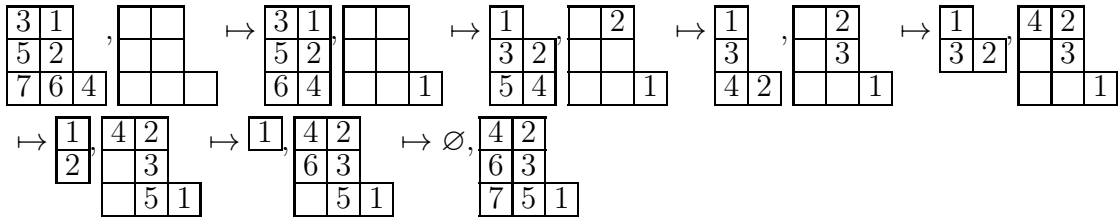
\begin{align*}
&\ytableausetup{mathmode,boxsize=1em}
\begin{ytableau}
3&1\\5&2\\7&6&4
\end{ytableau},
\begin{ytableau}
*(white) & *(white)\\*(white) & *(white)\\*(white) & *(white) & *(white)  
\end{ytableau} \mapsto 
\begin{ytableau}
3&1\\5&2\\6&4
\end{ytableau},
\begin{ytableau}
*(white) & *(white)\\*(white) & *(white)\\*(white) & *(white) & 1
\end{ytableau}\mapsto
\begin{ytableau}
1\\3&2\\5&4
\end{ytableau},
\begin{ytableau}
*(white) & 2\\*(white) & *(white)\\*(white) & *(white) & 1
\end{ytableau}\mapsto
\begin{ytableau}
1\\3\\4&2
\end{ytableau},
\begin{ytableau}
*(white) & 2\\*(white) & 3\\*(white) & *(white) & 1
\end{ytableau}\mapsto
\begin{ytableau}
1\\3&2
\end{ytableau},
\begin{ytableau}
4 & 2\\*(white) & 3\\*(white) & *(white) & 1
\end{ytableau}\\&\mapsto
\begin{ytableau}
1\\2
\end{ytableau},
\begin{ytableau}
4 & 2\\*(white) & 3\\*(white) & 5 & 1
\end{ytableau}\mapsto
\begin{ytableau}
1
\end{ytableau},
\begin{ytableau}
4 & 2\\6 & 3\\*(white) & 5 & 1
\end{ytableau}\mapsto
\varnothing ,
\begin{ytableau}
4 & 2\\6 & 3\\7 & 5 & 1
\end{ytableau}
\end{align*}
Thus, $e(T)$ is the $\SRT$ obtained in the last step above.
\end{Example}
Next, we will note down a property of our variant of the Robinson-Schensted algorithm. The proof is presented in the Appendix.
\begin{Lemma}\label{lemma: basic properties of RSK}
Let $w\in \mathfrak{S}_n$ and suppose further that $w\mapsto (P,Q)$ using the variant of Robinson-Schensted algorithm outlined earlier. Then
\begin{align*}
w^{-1} \mapsto (e(Q),e(P)).
\end{align*}
\end{Lemma}

\section{Background on noncommutative symmetric functions}\label{section: background nsym}
\subsection{Compositions}
A \textit{composition} of $n$ is an ordered sequence of non-negative integers $\alpha = (\alpha_1,\ldots,\alpha_k)$ such that $\sum_{i=1}^{k}\alpha_{i}$ equals $n$. The integers $\alpha_i$ are called the \textit{parts} of $\alpha$. Given a composition $\alpha=(\alpha_1,\ldots,\alpha_k)$, the \textit{size} $\lvert \alpha\rvert$ is defined to be $\sum_{i=1}^{k}\alpha_i$. The number of parts of $\alpha$ is called its \textit{length}, and is denoted by $l(\alpha)$. As was the case with partitions, $\varnothing$ denotes the unique composition of size 0 and length 0. 

If all parts of a composition $\alpha$ are positive, then the composition is called a \textit{strong} composition. If there are parts that equal $0$, then we call $\alpha$ a \textit{weak} composition. The notation $\alpha \vDash n$ denotes that $\alpha$ is a composition of $n$. We will mainly be using strong compositions (and henceforth will drop the adjective strong). If we do require weak compositions, we will state this explicitly. Given a composition $\alpha$, we denote by $\widetilde{\alpha}$ the partition obtained by sorting the parts of $\alpha$ in weakly decreasing order.

A composition $\alpha=(\alpha_1,\ldots,\alpha_k)$ of $n$ will be represented pictorially by a \textit{reverse composition diagram} that is a left-justified array of $n$ boxes with $\alpha_{i}$ boxes in row $i$. Our labeling of rows will follow the English convention. Hence the rows are labeled from top to bottom and the columns from left to right.

There exists a bijection between compositions of $n$ and subsets of $[n-1]$ that we recall now. Given a composition $\alpha=(\alpha_1,\ldots,\alpha_k)\vDash n$, we can associate a subset of $[n-1]$, called $\set (\alpha)=\{\alpha_1,\alpha_1+\alpha_2,\ldots,\alpha_1+\cdots +\alpha_{k-1}\}$. In the opposite direction, given a set $S=\{i_1< \cdots  <i_j\}\subseteq [n-1]$, we can associate a composition of $n$, called $\comp (S)=(i_1,i_2-i_1,\ldots, i_j-i_{j-1},n-i_j)$.

Finally, we define the \textit{refinement order} on compositions. Given compositions $\alpha$ and $\beta$, we say that $\alpha \succcurlyeq \beta$ if one obtains parts of $\alpha$ in order by adding together adjacent parts of $\beta$ in order. The composition $\beta$ is said to be a \textit{refinement} of $\alpha$, or equivalently, $\alpha$ is said to be a \textit{coarsening} of $\beta$.  

\begin{Example}
Let $\alpha=(1,3,1,1,2)\vDash 8$. Then $l(\alpha)=5$ and $\set (\alpha)=\{1,4,5,6\}\subseteq [7]$. The reverse composition diagram of $\alpha$ is as shown.
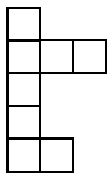
\begin{eqnarray*}
\ydiagram{1,3,1,1,2}
\end{eqnarray*}
Also $(4,1,3)\succcurlyeq (1,3,1,1,2)$.
\end{Example}

\subsection{Semistandard reverse composition tableaux }\label{neededlater}
Let $\alpha = (\alpha_1,\ldots, \alpha_l)$ and $\beta$ be compositions. Define a cover relation, $\lessdot _{c}$, on compositions in the following way.
\begin{eqnarray*}
\alpha \lessdot_{c} \beta \text{ if and only if }\left\lbrace \begin{array}{ll} \beta= (1,\alpha_1,\ldots,\alpha_l) & \text{or}\\ \beta= (\alpha_1,\ldots,\alpha_k +1,\ldots,\alpha_l) & \text{and $\alpha_i\neq \alpha_k$ for all $i<k$} \end{array}\right.
\end{eqnarray*}
The \emph{reverse composition poset} $\mathcal{L}_{c}$ is the poset on compositions where the partial order $<_{c}$ is obtained by taking the transitive closure of the cover relation $\lessdot_{c}$ above.
If $\beta <_{c} \alpha$, the \textit{skew reverse composition shape} $\alpha \cskew \beta$ is defined to be the array of boxes
\begin{eqnarray*}
\alpha\cskew \beta = \{(i,j): (i,j)\in \alpha, (i,j)\notin \beta \}
\end{eqnarray*}
where $\beta$ is drawn in the bottom left corner of $\alpha$. We refer to $\alpha$ as the \textit{outer shape} and to $\beta$ as the \textit{inner shape}. If the inner shape is $\varnothing$, instead of writing $\alpha \cskew \varnothing$, we just write $\alpha$ and refer to $\alpha$ as a \textit{straight shape}. The \emph{size} of the skew reverse composition shape $\alpha\cskew \beta$, denoted by $\vert \alpha\cskew\beta\vert$, is the number of boxes in the skew reverse composition shape, that is, $\vert \alpha\vert-\vert\beta\vert$. Now, we will define a semistandard reverse composition tableau.
\begin{Definition}
A \emph{semistandard reverse composition tableau} ($\SSRCT$) $\tau$ of \emph{shape} $\alpha \cskew \beta$ is a filling 
\begin{eqnarray*}
\tau: \alpha\cskew \beta \longrightarrow \mathbb{Z}^{+}
\end{eqnarray*}
that satisfies the following conditions
\begin{enumerate}
\item the rows are weakly decreasing from left to right,
\item the entries in the first column are strictly increasing from top to bottom,
\item if $i< j$ and $(j,k+1) \in \alpha \cskew \beta$ and either $(i,k)\in \beta$ or $\tau(i,k) \geq \tau(j,k+1)$ then either $(i,k+1)\in \beta$ or both $(i,k+1) \in \alpha \cskew \beta $ and $\tau(i,k+1) > \tau(j,k+1)$.
\end{enumerate}
\end{Definition}
The shape underlying an $\SSRCT$ $\tau$ will be denoted by $\sha (\tau)$. We will denote the set of $\SSRCT$s of shape $\alpha\cskew \beta$ by $\SSRCT(\alpha\cskew \beta)$

When considering an $\SSRCT$, the boxes of the inner shape are filled with bullets. This given, the third condition in the definition above is equivalent to the non-existence of the following configurations in the filling $\tau$.
\begin{eqnarray*}
\ytableausetup{mathmode,boxsize=1.3em}
\begin{array}{lllllllllllll}
\begin{ytableau}
\bullet & y\\ \none & \none[\vdots]\\\none & z
\end{ytableau}& \text{ } & \text{ } & \text{ } &
\begin{ytableau}
\bullet \\ \none & \none[\vdots]\\\none & z
\end{ytableau}& \text{ } & \text{ } & \text{ } &
\begin{ytableau}
x & y\\ \none & \none[\vdots]\\\none & z
\end{ytableau}& \text{ } & \text{ } & \text{ } &
\begin{ytableau}
x \\ \none & \none[\vdots]\\\none & z
\end{ytableau}\\
z\geq y>0&  \text{ } & \text{ } & \text{ } &z>0&  \text{ } & \text{ } & \text{ } &  x \geq z\geq y>0 & \text{ } & \text{ } & \text{ } & x\geq z>0
\end{array}
\end{eqnarray*}
The existence of such a configuration in a filling will be termed a \textit{triple rule violation}.
Our strategy to show that a given filling is an $\SSRCT$ will be to show that there are no triple rule violations therein.

The \textit{column reading word} of an $\SSRCT$ $\tau$ is the word obtained by reading the entries in each column in increasing order, where we traverse the columns from left to right. We denote it by $w_{\tau}$.

A \emph{standard reverse composition tableau} (abbreviated to $\SRCT$) is an $\SSRCT$ in which the filling is a bijection $\tau: \alpha\cskew\beta \longrightarrow [|\alpha\cskew \beta|]$. That is, in an $\SRCT$ each number from the set $\{1,2,\ldots, |\alpha\cskew\beta|\}$ appears exactly once. 
\begin{Example}\label{exampleskewssrctstraightsrct}
An $\SSRCT$ of shape $(2,5,4,2)\cskew (2,1,2)$ (left) and an $\SRCT$ of shape $(3,4,2,3)$.
\begin{eqnarray*}
\ytableausetup{mathmode,boxsize=1.3em}
\begin{ytableau}
4 &3\\ \bullet & \bullet &7 &5 &1\\\bullet & 7 &6 &2\\ \bullet & \bullet
\end{ytableau}\hspace{15mm}
\begin{ytableau}
 6& 4& 1\\8& 7& 5& 2\\ 10& 3\\ 12& 11& 9
\end{ytableau}
\end{eqnarray*}
The column reading word of the $\SSRCT$ $\tau$ on the left is $w_{\tau}=4\hspace{2mm} 37\hspace{2mm}67\hspace{2mm}25\hspace{2mm}1$.
\end{Example}

Also associated with an $\SRCT$ is its \textit{descent set}. Given an $\SRCT$ $\ctau$ of shape $\alpha\vDash n$, its descent set, denoted by $\des (\ctau)$, is defined to be the set of all integers $i$ such that $i+1$ lies weakly to the right of $i$ in $\ctau$. Note that $\des (\ctau)$ is a subset of $[n-1]$. The \textit{descent composition} of $\ctau$, denoted by $\comp (\ctau)$, is the composition of $n$ associated with $\des (\ctau)$. As an example, the descent set of the $\SRCT$ in Example \ref{exampleskewssrctstraightsrct} is $\{1,3,4,6,8,10\}\subseteq [11]$. Hence the associated descent composition is $(1,2,1,2,2,2,2)\vDash 12$.

Given a composition $\alpha$, there exists a unique $\SRCT$ $\tau_{\alpha}$ so that the underlying shape of $\tau_{\alpha}$ is $\alpha$ and $\comp(\tau_{\alpha})=\alpha$. This tableau is called the \textit{canonical tableau} of shape $\alpha$. To construct $\tau_{\alpha}$ given $\alpha = (\alpha_1,\ldots,\alpha_k)$, place consecutive integers from $1+\sum_{i=1}^{r-1}\alpha_i$ to $\sum_{i=1}^{r}\alpha_i$ in that order from right to left in the $r$-th row from the top in the reverse composition diagram of $\alpha$ for $1\leq r\leq k$.

\begin{Example}\label{example: canonical tableau}
Let $\alpha=(3,4,2,3)$. Then $\tau_{\alpha}$ is as shown below.
$$
\ytableausetup{mathmode,boxsize=1.3em}
\begin{ytableau}
3&2&1\\7&6&5&4\\9&8\\12 &11 & 10
\end{ytableau}
$$
\end{Example}

Next, we discuss the relationship between $\SSRCT$s and $\SSRT$s. We will start by defining a generalization of the bijection defined by Mason in \cite{Mason-rho} (the generalization itself is presented in \cite{LMvW}). Let $\SSRCT(-\cskew\alpha)$ denote the set of all $\SSRCT$s with inner shape $\alpha$, and let $\SSRT(-/\widetilde{\alpha})$ denote the set of $\SSRT$s with inner shape $\widetilde{\alpha}$ (which might be empty). Then the \emph{generalized $\rho$ map}
\begin{eqnarray*}
\rho_{\alpha}: \SSRCT(-\cskew\alpha) \longrightarrow \SSRT(-/\widetilde{\alpha})
\end{eqnarray*}
is defined as follows. Given an $\SSRCT$ $\tau$, write the entries in each column in decreasing order from bottom to top, and bottom justify the new columns thus obtained on the inner shape $\widetilde{\alpha}$. This yields $\rho_{\alpha}(\tau)$.

We give the description of the inverse map $\rho_{\alpha}^{-1}$ next. Let $T$ be an $\SSRT$.
\begin{enumerate}
\item Take the first column of the $\SSRT$, and place the entries of this column above the first column of the inner shape $\alpha$ in increasing order from top to bottom.
\item Now take the set of entries in the second column in decreasing order and place them in the row with smallest index so that either
\begin{itemize}
\item the box to the immediate left of the number being placed is filled and the entry therein is weakly greater than the number being placed
\item the box to the immediate left of the number being placed belongs to the inner shape $\alpha$.
\end{itemize}
\end{enumerate}
As a matter of convention, instead of writing $\rho_{\varnothing}$ or $\rho_{\varnothing}^{-1}$, we will use $\rho$ or $\rho^{-1}$ respectively.
\begin{Example}
An $\SSRCT$ of shape $(2,5,4,2)\cskew (2,1,2)$ (left), and the corresponding $\SSRT$ of shape $(5,4,2,2)/(2,2,1)$ (right).
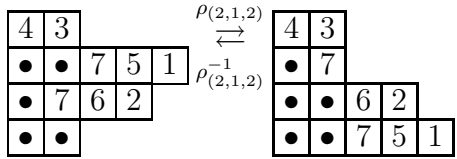
\begin{eqnarray*}
\ytableausetup{mathmode,boxsize=1.1em}
\begin{ytableau}
4 &3\\ \bullet & \bullet &7 &5 &1\\\bullet & 7&6 &2\\ \bullet & \bullet
\end{ytableau}
\mathrel{\mathop{\rightleftarrows}^{\mathrm{\rho_{(2,1,2)}}}_{\mathrm{\rho_{(2,1,2)}^{-1}}}} 
\ytableausetup{mathmode,boxsize=1.1em}
\begin{ytableau}
4&3\\ \bullet &7\\\bullet&\bullet&6&2\\\bullet&\bullet&7&5&1
\end{ytableau}
\end{eqnarray*}
\end{Example}

\subsection{Noncommutative symmetric functions}\label{NSym}
An algebra closely related to $\Lambda$ is the algebra of \textit{noncommutative symmetric functions} $\mathbf{NSym}$, introduced in \cite{GKLLRT}. It is the free associative algebra $\mathbb{Q}\langle \nch_1,\nch_2 ,\ldots \rangle$ generated by a countably infinite number of indeterminates $\nch_k$ for $k\geq 1$. Assigning degree $k$ to $\nch_k$, and extending this multiplicatively allows us to endow $\mathbf{NSym}$ with the structure of a graded algebra. A natural basis for the degree $n$ graded component of $\mathbf{NSym}$, denoted by $\mathbf{NSym}^{n}$, is given by the \textit{noncommutative complete homogeneous symmetric functions}, $\{\mathbf{h}_{\alpha}=\mathbf{h}_{\alpha_1}\cdots \mathbf{h}_{\alpha_k}:\alpha=(\alpha_1,\ldots,\alpha_k)\vDash n\}$.
The link between $\Lambda$ and $\mathbf{NSym}$ is made manifest through the \textit{forgetful} map, $\chi: \mathbf{NSym}\to \Lambda$, defined by mapping $\nch_i$ to $h_i$ and extending multiplicatively. Thus, the images of elements of $\mathbf{NSym}$ under $\chi$ are elements of $\Lambda$.

$\mathbf{NSym}$ has another important basis called the noncommutative ribbon Schur basis. We will denote the \textit{noncommutative ribbon Schur function} indexed by a composition $\beta$ by $\rib_{\beta}$. The following \cite[Proposition 4.13]{GKLLRT} can be taken as the definition of the noncommutative ribbon Schur functions.
\begin{eqnarray*}
\rib_{\beta}&=&\displaystyle\sum_{\alpha \succcurlyeq \beta}(-1)^{l(\beta)-l(\alpha)}\nch_{\alpha}
\end{eqnarray*}
Using the above basis, we can define the noncommutative Schur functions.

\subsection{Noncommutative Schur functions}
We will now describe a distinguished basis for $\mathbf{NSym}$, introduced in \cite{BLvW}, called the basis of \textit{noncommutative Schur functions}. These functions are naturally indexed by compositions, and the noncommutative Schur function indexed by a composition $\alpha$ will be denoted by $\ncs_{\alpha}$. They are defined implicitly using the relation
\begin{eqnarray*}\label{def:implicitnoncommutativeschur}
\rib_{\beta}=\displaystyle\sum_{\alpha \vDash \vert\beta\vert} d_{\alpha\beta}\ncs_{\alpha}
\end{eqnarray*}
where $d_{\alpha\beta}$ is the number of $\SRCT$s of shape $\alpha$ and descent composition $\beta$.

The noncommutative Schur function $\ncsa$ satisfies the following important property \cite[Equation 2.12]{BLvW}. \begin{align*}\chi(\ncsa)=s_{\widetilde{\alpha}}\end{align*}
Thus, the noncommutative Schur functions are noncommutative lifts of the Schur functions in $\mathbf{NSym}$ and, in fact, share many properties with the Schur functions. The interested reader should refer to \cite{BLvW, LMvW} for a detailed study of these functions.
For our purposes, we require the noncommutative Littlewood-Richardson rule proved in \cite{BLvW}.
\begin{Theorem}\cite[Theorem 3.5]{BLvW}\label{theorem: noncommutative LR rule}
Given compositions $\alpha$ and $\beta$, the product $\ncsa \cdot \ncsb$ can be expanded in the basis of noncommutative Schur functions as follows
\begin{align*}
\ncsa\cdot \ncsb = \displaystyle\sum_{\gamma}C_{\alpha\beta}^{\gamma}\ncsg
\end{align*}
where $C_{\alpha\beta}^{\gamma}$ counts the number of $\SRCT$s of shape $\gamma\cskew \beta$ that rectify to the canonical tableau of shape $\alpha$.
\end{Theorem}
We should clarify what it means to rectify for $\SRCT$s. We say that an $\SRCT$ $\tau_1$ of shape $\gamma\cskew \beta$ rectifies to a tableau $\tau_2$ of straight shape $\alpha$ if $\rho^{-1}(P(w_{\tau_1}))=\tau_2$.
\section{Box removing operators and jeu de taquin operators on compositions}\label{section: operators on compositions}
In this section, we will define some operators on compositions that will play an important role later.

The \textit{box removing operators} on compositions, denoted by $\down_i$ for $i\geq 1$, have the following description: $\down_i(\alpha)=\alpha'$ where $\alpha'$ is the composition obtained by subtracting 1 from the rightmost part of length $i$ in $\alpha$ and omitting any $0$s that might result in so doing. If there is no such part, then $\down_i(\alpha)=0$.
\begin{Example}
Let $\alpha=(2,1,2)$. Then $\down_1(\alpha)=(2,2)$ and $\down_2(\alpha)=(2,1,1)$.
\end{Example}
Given a nonempty set of positive integers $S=\{a_1<a_2<\cdots <a_r\}$, we will also be needing the following \text{p}roduct of box removing operators.
\begin{eqnarray*}
\downrow_{S}=\down_{a_1}\down_{a_2}\cdots \down_{a_r}
\end{eqnarray*}
Note that the product above is not commutative, and we will consider $\downrow_S$ as operators $\down_{a_r}$, $\down_{a_{r-1}}$, $\ldots$, $\down_{a_1}$ applied in succession to a composition, in that order. As a matter of convention, $\downrow_{\varnothing}$ is to be interpreted as the identity map.
Of particular interest to us is the operator $\downrow_{[i]}$ for $i\geq 1$, where $[i]=\{1,\ldots,i\}$. We define [0] to be the empty set $\varnothing$. Then $\downrow_{[0]}$ is the identity map.

Next, define an operator $a_i$ that appends a part of length $i$ to a composition $\alpha$ at the end, for $i\geq 1$. For example, $a_2((2,1,3))=(2,1,3,2)$. Notice also that $a_j\down_2((3,5,1))=0$ for any positive integer $j$, as $\down_2((3,5,1))=0$.

With the definitions of $a_i$ and $\downrow_{[i]}$, we can define the \textit{jeu de taquin operators} (henceforth abbreviated to \textit{jdt operators}), $\up_i$ for $i\geq 1$, as follows.
\begin{eqnarray*}
\up_i=a_i\downrow_{[i-1]}
\end{eqnarray*}

\begin{Example}
We will compute $\up_4(\alpha)$ where $\alpha=(6,3,2,3,1,5,4)$. Firstly, $\down_1\down_2\down_3(\alpha)=(6,3,2,1,5,4)$, and thus $a_4\downrow_{[3]}(\alpha)=(6,3,2,1,5,4,4)$.
\end{Example}
\begin{Remark}\label{remark: up operator on partition}
Let $\alpha$ be a composition and $\lambda =\widetilde{\alpha}$. Let $\beta = u_i(\alpha)$ for some positive integer $i$. Then $\widetilde{\beta}$ is the partition obtained by adding a box at the addable node to $\lambda$ in column $i$.
\end{Remark}

\section{Statement of main results}\label{section: Results}
In this section, we will state our algorithm for our analogue of the classical backward jeu de taquin slide. We will work under the following assumptions in this section and for all the proofs in the coming sections, unless otherwise stated.
\begin{itemize}
\item $i$ will denote a positive integer $\geq 1$.
\item $T$ will refer to a fixed $\SSRT$ of shape $\lambda$.
\item The $\SSRCT$ $\rho^{-1}(T)$ will be denoted by $\tau$ and we will assume that it has shape $\alpha$ where $\alpha$ clearly satisfies $\widetilde{\alpha}=\lambda$.
\item There exists an addable node in column $i+1$ of $\lambda$. 
\end{itemize}
\subsection{Backward jeu de taquin slides for SSRCTs}\label{subsection: results 1}
With the notation from above, define a positive integer $r_1$ to be such that such that $\alpha_{r_1}$ is the bottommost part of length $i$ in $\alpha$. This given define positive integers $r_j$ for $j\geq 2$ recursively subject to the following conditions.
\begin{enumerate}[I.]
\item $r_j<r_{j-1}$.
\item $r_j$ is the greatest positive integer such that $\tau (r_{j},i)>\tau(r_{j-1},i)\geq \tau (r_{j},i+1)$. If the box in position $(r_{j},i+1)$ does not belong to $\tau$, we let $\tau(r_{j},i+1)$ equal 0.
\end{enumerate}
Clearly the $r_j$s defined above are finitely many. Let $r_k$ denote the minimum element and
$S=\{r_1>\cdots > r_k\}$.
Next we give an algorithm $\phi_{i+1}$ that takes $\tau$ as input and outputs an $\SSRCT$ $\phi_{i+1}(\tau)$ using the set $S$ above. 
\begin{Algorithm}\label{results algorithm1}
\hfill
\begin{enumerate}[(i)]
\item For j=$k$ down to $1$ in that order,
\begin{itemize}
\item Replace the entry in box $(r_{j},i)$ by the entry in box $(r_{j-1},i)$. The entry $\tau(r_k,i)$ \emph{exits} the tableau in every iteration. 
\item Remove the box in position $(r_{1},i)$. 
\end{itemize}
\item Denote the resulting filling by $\phi_{i+1}(\tau)$.
\end{enumerate}
\end{Algorithm}
We will use the above algorithm and give a procedure on $\SSRCT$s of straight shape, referred to as $\mu_i$, and this is analogous to the procedure $\jdt_i$ on $\SSRT$s of straight shape. Our input is again $\tau$ and the output is an $\SSRCT$ $\mu_{i}(\tau)$.
\begin{Algorithm}\label{results algorithm2}
\hfill
\begin{enumerate}[(i)]
\item If $i\geq 2$, compute first the $\SSRCT$ $\phi_{2}\circ \cdots \circ\phi_{i}(\tau)$. Denote this by $\tau'$. If $i=1$, define $\tau'$ to be $\tau$. Furthermore, store the entries that exit in an array $H$, in the case $i\geq 2$.
\item Let $\sha(\tau')=\gamma=(\gamma_1,\ldots,\gamma_k)$. Let $\delta = (\gamma_1,\ldots,\gamma_k,i)$.
\item Consider a filling of the skew reverse composition shape $\delta \cskew (1)$ where the top $k$ rows are filled according to $\tau'$, and the last row is filled by the entries in $H$ in decreasing order from left to right.
\item The filling of the skew reverse composition shape $\delta \cskew (1)$ obtained above is defined to be $\mu_{i}(\tau)$.
\end{enumerate}
\end{Algorithm}
Now we will state a theorem that gives credence to our claim that $\mu_i$ is analogous to $\jdt_i$.
\begin{Theorem}\label{theorem: results backwardjdtslideonssrct}
With the notation as before, the following diagram commutes.
\begin{eqnarray*}
\xymatrix{ \text{$\SSRT$\emph{s}} \ar[d]_{\jdt_{i}} \ar[r]^{\rho^{-1}} & \text{$\SSRCT$\emph{s}}\ar[d]^{\mu_{i}}\\ \text{$\SSRT$\emph{s}} \ar[r]_{\rho_{(1)}^{-1}} & \text{$\SSRCT$\emph{s}}}
\end{eqnarray*}
\end{Theorem}
\begin{Example}
Given below is an $\SSRT$ $T$ (left) and the $\SSRCT$ $\tau= \rho^{-1}(T)$ (right).
\begin{eqnarray*}
T=\ytableausetup{centertableaux,boxsize=1.25em}
\begin{ytableau}
8\\19 & 9 & 6\\24 & 17 & 7 & 3\\26 & 23 & 18 & 10\\37 & 28 & 25 & 20\\43 & 35 & 27 & 22 & 16 & 12 & 1\\44 & 38 & 36 & 29 & 21 & 15 & 4 & 2\\47 & 42 & 41 & 33 & 32 & 31 & 13 & 5\\48 & 46 & 45 & 40 & 39 & 34 & 30 & 14 & 11
\end{ytableau}
\tau =\ytableausetup{centertableaux,boxsize=1.25em}
\begin{ytableau}
8\\19 & 17 & 7 & 3\\24 & 23 & 18 & 10\\26 & 9 & 6\\37 & 35 & 27 & 22 & 21 & 15 & 13 & 5\\43 & 42 & 41 & 40 & 39 & 34 & 30 & 14 & 11\\44 & 38 & 36 & 33 & 32 & 31 & 4 & 2\\47 & 46 & 45 & 29 & 16 & 12 & 1\\48 & 28 & 25 & 20
\end{ytableau}
\end{eqnarray*}

We will compute $\mu_{9}(\tau)$. We start by computing $\phi_{2}\circ \cdots \circ \phi_{9}(\tau)$. To comprehensively illustrate the theorem, this is shown step by step below. The entries in rows $r_j$ for $1\leq j\leq k$ that we need in applying Algorithm \ref{results algorithm1} are depicted in green, except the entry that exits, which is depicted in red.
\begin{eqnarray*}
\begin{array}{ccc}
\begin{ytableau}
8\\19 & 17 & 7 & 3\\24 & 23 & 18 & 10\\26 & 9 & 6\\37 & 35 & 27 & 22 & 21 & 15 & 13 & *(red) 5\\43 & 42 & 41 & 40 & 39 & 34 & 30 & 14 & 11\\44 & 38 & 36 & 33 & 32 & 31 & 4 & *(green) 2\\47 & 46 & 45 & 29 & 16 & 12 & 1\\48 & 28 & 25 & 20
\end{ytableau} & 
\begin{ytableau}
8\\19 & 17 & 7 & 3\\24 & 23 & 18 & 10\\26 & 9 & 6\\37 & 35 & 27 & 22 & 21 & 15 & *(red) 13 & 2\\43 & 42 & 41 & 40 & 39 & 34 & 30 & 14 & 11\\44 & 38 & 36 & 33 & 32 & 31 & *(green)4\\47 & 46 & 45 & 29 & 16 & 12 & *(green) 1\\48 & 28 & 25 & 20
\end{ytableau} &
\begin{ytableau}
8\\19 & 17 & 7 & 3\\24 & 23 & 18 & 10\\26 & 9 & 6\\37 & 35 & 27 & 22 & 21 & 15 & 4 & 2\\43 & 42 & 41 & 40 & 39 & *(red) 34 & 30 & 14 & 11\\44 & 38 & 36 & 33 & 32 & *(green) 31 & 1\\47 & 46 & 45 & 29 & 16 & *(green) 12\\48 & 28 & 25 & 20
\end{ytableau}
\\\text{ }& \text{ } &\text{ }\\ \phi_{9}(\tau) & \phi_{8}\circ\phi_{9}(\tau) & \phi_{7}\circ \phi_{8} \circ\phi_{9}(\tau)\\
\text{ }& \text{ } &\text{ }\\
\begin{ytableau}
8\\19 & 17 & 7 & 3\\24 & 23 & 18 & 10\\26 & 9 & 6\\37 & 35 & 27 & 22 & 21 & 15 & 4 & 2\\43 & 42 & 41 & 40 & *(red) 39 & 31 & 30 & 14 & 11\\44 & 38 & 36 & 33 & *(green) 32 & 12 & 1\\47 & 46 & 45 & 29 & *(green) 16\\48 & 28 & 25 & 20
\end{ytableau} &
\begin{ytableau}
8\\19 & 17 & 7 & 3\\24 & 23 & 18 & 10\\26 & 9 & 6\\37 & 35 & 27 & 22 & 21 & 15 & 4 & 2\\43 & 42 & 41 & *(red) 40 & 32 & 31 & 30 & 14 & 11\\44 & 38 & 36 & *(green) 33 & 16 & 12 & 1\\47 & 46 & 45 & *(green) 29\\48 & 28 & 25 & *(green) 20
\end{ytableau} &
\begin{ytableau}
8\\19 & 17 & 7 & 3\\24 & 23 & 18 & 10\\26 & 9 & 6\\37 & 35 & 27 & 22 & 21 & 15 & 4 & 2\\43 & 42 & 41 & 33 & 32 & 31 & 30 & 14 & 11\\44 & 38 & 36 & 29 & 16 & 12 & 1\\47 & 46 & *(red) 45 & 20\\48 & 28 & *(green) 25
\end{ytableau}\\
\text{ } & \text{ } & \text{ }\\ 
\phi_{6}\circ \cdots \circ\phi_{9}(\tau) & \phi_{5}\circ \cdots \circ\phi_{9}(\tau) & \phi_{4}\circ \cdots \circ\phi_{9}(\tau) \\
\text{ } & \text{ } & \text{ }
\\
\begin{ytableau}
8\\19 & 17 & 7 & 3\\24 & 23 & 18 & 10\\26 & 9 & 6\\37 & 35 & 27 & 22 & 21 & 15 & 4 & 2\\43 & 42 & 41 & 33 & 32 & 31 & 30 & 14 & 11\\44 & 38 & 36 & 29 & 16 & 12 & 1\\47 & *(red)46 & 25 & 20\\48 & *(green)28 
\end{ytableau} &
\begin{ytableau}
8\\19 & 17 & 7 & 3\\24 & 23 & 18 & 10\\26 & 9 & 6\\37 & 35 & 27 & 22 & 21 & 15 & 4 & 2\\43 & 42 & 41 & 33 & 32 & 31 & 30 & 14 & 11\\44 & 38 & 36 & 29 & 16 & 12 & 1\\47 & 28 & 25 & 20\\*(red)48 
\end{ytableau} & \text{ }
\\
\text{ } & \text{ } & \text{ }\\
\phi_{3}\circ \cdots \circ \phi_{9}(\tau) & \phi_2\circ \cdots \circ \phi_9(\tau) & \text{ }
\end{array}
\end{eqnarray*}
Thus, finally we have the following. 
\begin{eqnarray*}
\begin{array}{cc}
\ytableausetup{centertableaux}
\begin{ytableau}
8\\19 & 17 & 7 & 3\\24 & 23 & 18 & 10\\26 & 9 & 6\\37 & 35 & 27 & 22 & 21 & 15 & 4 & 2\\43 & 42 & 41 & 33 & 32 & 31 & 30 & 14 & 11\\44 & 38 & 36 & 29 & 16 & 12 & 1\\47 & 28 & 25 & 20\\ \none & *(red)48 & *(red) 46 & *(red) 45 & *(red) 40 & *(red) 39 & *(red) 34 & *(red) 13 & *(red) 5
\end{ytableau} &
\begin{ytableau}
8\\19 & 9 & 6\\24 & 17 & 7 & 3\\26 & 23 & 18 & 10\\37 & 28 & 25 & 20\\43 & 35 & 27 & 22 & 16 & 12 & 1\\44 & 38 & 36 & 29 & 21 & 15 & 4 & 2\\47 & 42 & 41 & 33 & 32 & 31 & 30& 13 & 5\\\none & 48 & 46 & 45 & 40 & 39 & 34 & 14 & 11
\end{ytableau}
\\
\text{ } &\text{ }\\
 \mu_{9}(\tau) & \jdt_{9}(T)
\end{array}
\end{eqnarray*}
One can see that $\rho_{(1)}(\mu_{9}(\tau))=\jdt_{9}(T)$.
\end{Example}
Computing $\mu_i$ applied to $\SSRCT$s of skew reverse composition shape is very similar, and is discussed in Section \ref{section: slides for skew shape}.
\subsection{Right Pieri rule}\label{subsection: results 2}
In this subsection, we will state our right Pieri rule using the jdt operators and then give an example. The proof is given in Section \ref{section: right pieri rule}.
\begin{Theorem}[Right Pieri rule]\label{theorem: right pieri rules}
Let $\alpha$ be a composition and $n$ a positive integer. Then
\begin{eqnarray*}
\ncsa\cdot \ncs_{(n)}=\displaystyle\sum_{\substack{\beta\vDash |\alpha|+n, u_{i_n}\cdots u_{i_1}(\alpha)=\beta\\i_n>\cdots >i_1}} \ncsb ,\\
\ncsa\cdot \ncs_{(1^n)}=\displaystyle\sum_{\substack{\beta\vDash |\alpha|+n, u_{i_n}\cdots u_{i_1}(\alpha)=\beta\\i_n\leq \cdots \leq i_1}} \ncsb .
\end{eqnarray*}
\end{Theorem}
In the statement above, $(1^n)$ denotes the composition with $n$ parts equalling $1$ each.
\begin{Example}
Let $\alpha=(1,4,2)$. We will compute $\ncsa\cdot \ncs_{(3)}$ first. Thus, we are interested in finding positive integers $i_1<i_2<i_3$ such that $u_{i_3}u_{i_2}u_{i_1}(\alpha)$ is a composition of size $10$. One can see that we have the following choices.
\begin{eqnarray*}
\begin{array}{lll}
u_3u_2u_1(\alpha)=(1,4,2,3) & u_5u_2u_1(\alpha)=(1,2,2,5) & u_4u_3u_1(\alpha)=(1,4,1,4)\\u_5u_3u_1(\alpha)=(1,3,1,5) & u_6u_5u_1(\alpha)=(1,2,1,6) & u_4u_3u_2(\alpha)=(4,2,4)\\u_5u_3u_2(\alpha)=(3,2,5) & u_6u_5u_2(\alpha)=(2,2,6) & u_5u_4u_3(\alpha)=(1,4,5) \\ u_6u_5u_3(\alpha)=(1,3,6) & u_7u_6u_5(\alpha)=(1,2,7) & \text{ }
\end{array}
\end{eqnarray*}
This gives us the following expansion for $\ncs_{(1,4,2)}\cdot \ncs_{(3)}$.
\begin{eqnarray*}
\ncs_{(1,4,2)}\cdot \ncs_{(3)}=&& \ncs_{(1,4,2,3)}+\ncs_{(1,2,2,5)}+\ncs_{(1,4,1,4)}+\ncs_{(1,3,1,5)}+\ncs_{(1,2,1,6)}+\ncs_{(4,2,4)}\\&&+\ncs_{(3,2,5)}+\ncs_{(2,2,6)}+\ncs_{(1,4,5)}
+\ncs_{(1,3,6)}+\ncs_{(1,2,7)}
\end{eqnarray*}
Next, we compute $\ncsa\cdot \ncs_{(1,1,1)}$. Then we are interested in finding positive integers $i_1\geq i_2\geq i_3$ such that $u_{i_3}u_{i_2}u_{i_1}(\alpha)$ is a valid composition of size $10$. Thus, we have the following choices.
\begin{eqnarray*}
\begin{array}{lll}
u_1u_1u_1(\alpha)=(1,4,2,1,1,1) & u_1u_1u_2(\alpha)= (4,2,2,1,1) & u_1u_1u_3(\alpha)=(1,4,3,1,1)\\u_1u_1u_5(\alpha)=(1,2,5,1,1) & u_1u_2u_3(\alpha)=(4,3,2,1) & u_1u_2u_5(\alpha)=(2,5,2,1)\\u_1u_3u_5(\alpha)=(1,5,3,1) & u_2u_3u_5(\alpha)=(5,3,2) & \text{ }
\end{array}
\end{eqnarray*}
This gives us the following expansion for $\ncs_{(1,4,2)}\cdot \ncs_{(1,1,1)}$.
\begin{eqnarray*}
\ncs_{(1,4,2)}\cdot \ncs_{(1,1,1)}=&& \ncs_{(1,4,2,1,1,1)}+\ncs_{(4,2,2,1,1)}+\ncs_{(1,4,3,1,1)}+\ncs_{(1,2,5,1,1)}+\ncs_{(4,3,2,1)}\\&&+\ncs_{(2,5,2,1)}
+\ncs_{(1,5,3,1)}+\ncs_{(5,3,2)}
\end{eqnarray*}
\end{Example}

\section{Proof of validity of algorithm $\mu_i$}\label{section: backward jdt slides straight shape}
In this section, we will give proofs of validity for the results in Subsection \ref{subsection: results 1}. But prior to that, we will require certain results on $\SSRT$s. As stated earlier, we will be using the notation established at the start of Section \ref{section: Results} for the proofs in this section. 

A particular entry of $T$ that plays a significant role in what follows is the first entry that moves horizontally while computing $\jdt_{i+1}(T)$. This entry is $T_{j,i}$ where $j$ is the largest integer such that $T_{j,i} < T_{j-1,i+1}$. (We assume $T_{0,s}=\infty$ for all positive integers $s$.) We shall denote this entry by $f_{i+1}(T)$ (the reason behind the name being that it is the \textit{f}irst entry that moves horizontally). We will use this in the next couple of lemmas, whose proofs are straightforward.
\begin{Lemma}
Let $f_{i+1}(T)=T_{r,i}$. Let $T'$ be the tableau defined below.
\begin{eqnarray*}
&T'_{p,q}=T_{p,q} \text{ if } q\neq i\\
&T'_{p,i}=T_{p,i} \text{ if } p<r\\
&T'_{p,i}=T_{p+1,i} \text{ if } p\geq r.
\end{eqnarray*}
Then $T'$ is an $\SSRT$.
\end{Lemma}
\begin{Remark}
In words, $T'$ is obtained from $T$ by first removing $T_{r,i}$ and sliding all entries above it in the $i$-th column down by one row. Note that the rest of the tableau is left unchanged.
\end{Remark}
\begin{proof}
The only thing we need to check is that the entries in each row of $T'$ are
weakly decreasing when read from left to right. More specifically, it
follows from the preceding remark that we only need to check if $T'_{j,i} \geq T'_{j,i+1}$ and $T'_{j,i-1} \geq T'_{j,i}$ (assuming $i\geq 2$ for this case) for $j\geq r$. 

The former inequality is clear since $T'_{j,i}=T_{j+1,i}$ and $T'_{j,i+1}=T_{j,i+1}$ for $j\geq r$ and by our hypothesis $T_{r,i}$ is the first entry that moves horizontally when a backward jeu de taquin slide is initiated from the $i+1$-th column. 

Now we will show that $T'_{j,i-1} \geq T'_{j,i}$ for $j\geq r$. We have 
$T'_{j,i}=T_{j+1,i}$ and $T'_{j,i-1}=T_{j,i-1}$. But it is clear that $T_{j,i-1} > T_{j+1,i}$ in any $\SSRT$ $T$. Hence the claim follows.
\end{proof}

We will denote the $\SSRT$ $T'$ above by $\tjdt _{i+1}(T)$ (the reason behind the name $\tjdt$ is that it stands for a `truncated' jeu de taquin slide).
Furthermore, we define $\tjdt_{1}(T)=T$, as no entries move horizontally when a backward jeu de taquin slide is initiated from the first column.
\begin{Example}
Let 
\begin{eqnarray*}
T=
\ytableausetup{mathmode,boxsize=1.2em}
\begin{ytableau}
6\\7 & 6\\ 8 & 7 & 4\\10 & 8 & 5\\11 &9 &9 &8
\end{ytableau}.
\end{eqnarray*}
If we consider a backward jeu de taquin slide from the addable node in column $3$, then we have $f_3(T)=T_{2,2}=8$. Thus
\begin{eqnarray*}
\emph{tjdt}_{3}(T)=
\ytableausetup{mathmode,boxsize=1.2em}
\begin{ytableau}
6\\7 \\ 8 & 6 & 4\\10 & 7 & 5\\11 &9 &9 &8
\end{ytableau}.
\end{eqnarray*}
\end{Example}
In the computation of $\jdt_{i+1}(T)$, clearly $i$ entries in $T$ move horizontally. It is crucial for us to know these entries. Define
\begin{eqnarray*}
H_{i+1}(T)&=&\text{ multiset of entries in }T\text{ that move horizontally.}
\end{eqnarray*}
We define $H_{1}(T)$ to be the empty set.
With this definition at hand, we state a second lemma.
\begin{Lemma}
\begin{eqnarray*}
H_{i}(\tjdt_{i+1}(T))\cup \{f_{i+1}(T)\}&=&H_{i+1}(T).
\end{eqnarray*}
In the equality above, we are considering union of multisets.
\end{Lemma}
\begin{proof}
Let $T'=\tjdt_{i+1}(T)$ and
$f_{i+1}(T)=T_{r,i}$. Our construction of $T'$ implies that there is an addable node in column
$i$ in the partition shape underlying $T'$. Hence, we can initiate a backward jeu de taquin slide from the
$i$-th column of $T'$ and $H_{i}(T')$ is well-defined.

If $i=1$, then we have nothing to prove. Hence, we can assume that $i\geq 2$. We will show that no entry of the form $T'_{j,i-1}$ with $j>r$ moves horizontally when a backward jeu de taquin slide is initiated from column $i$ in $T'$, as this clearly suffices to establish the claim. 

Suppose this was not the case. Then there exists $j>r$ such that 
\begin{align*}
T'_{j,i-1}<T'_{j-1,i}.
\end{align*}
Since $T'_{j,i-1}=T_{j,i-1}$ and $T'_{j-1,i}=T_{j,i}$ for $j>r$, the inequality above implies
\begin{align*}
T_{j,i-1}<T_{j,i}.
\end{align*}
But this contradicts the fact that $T$ is an $\SSRT$. 
\end{proof}
\begin{Example}
Let 
\begin{eqnarray*}
T=
\ytableausetup{mathmode,boxsize=1.2em}
\begin{ytableau}
6\\7 & 6\\ 8 & 7 & *(green)2\\10 & *(green)8 & 5&3\\*(green)11 &9 &9 &8
\end{ytableau}.
\end{eqnarray*}
If we consider a backward jeu de taquin slide from the addable node in
column $4$, then the entries in the boxes coloured green move horizontally. Thus, we have
\begin{eqnarray*}
H_{4}(T)=\{2,8,11\} \text{ and } f_{4}(T)=2.
\end{eqnarray*}
We also have that
\begin{eqnarray*}
\emph{tjdt}_{4}(T)=
\ytableausetup{mathmode,boxsize=1.2em}
\begin{ytableau}
6\\7 & 6\\ 8 & 7\\10 & *(green)8 & 5&3\\*(green)11 &9 &9 &8
\end{ytableau}.
\end{eqnarray*}
Now if one does a backward jeu de taquin slide on the tableau above
from the third column, then the entries coloured green move
horizontally, therefore
$H_3(\emph{tjdt}_{4}(T))=\{8,11\}$. Thus, we can verify in this case that 
\begin{eqnarray*}
H_4(T)=H_{3}(\emph{tjdt}_{4}(T))\cup f_{4}(T).
\end{eqnarray*}
\end{Example}

The lemma above allows us to recursively compute all the entries that move horizontally when a backward jeu de taquin slide is performed from an addable node in column $i+1$. We already know how to obtain $\tjdt_{i+1}(T)$ given an $\SSRT$ $T$. Before we define a backward jeu de taquin slide for $\SSRCT$s, let us consider the relatively easier question of computing $\rho^{-1}(\tjdt_{i+1}(T))$ given $\tau$. This is precisely what the algorithm $\phi_{i+1}$ described in Subsection \ref{subsection: results 1} achieves.
We will establish that $\phi_{i+1}$ maps $\SSRCT$s to $\SSRCT$s such that the following diagram commutes.
\begin{eqnarray*}
\xymatrix{ \text{$\SSRT$s} \ar[d]_{\tjdt_{i+1}} \ar[r]^{\rho^{-1}} & \text{$\SSRCT$s}\ar[d]^{\phi_{i+1}}\\ \text{$\SSRT$s} \ar[r]^{\rho^{-1}} & \text{$\SSRCT$s}}
\end{eqnarray*}

To allow for cleaner proofs, we will now describe how we will think of $T$ and $\tau$ diagrammatically. We will think of $T$ as aligned to the left in a rectangular Ferrers diagram where 
\begin{eqnarray*}
\text{number of rows, } n &=& l(\lambda),\\
\text{number of columns} &\geq & \lambda_1+1.
\end{eqnarray*} 
The empty boxes of this Ferrers diagram are assumed to be filled with $0$s. 
We will call this diagram $T_{\square}$. Note that the $0$s do not play any role as far as doing the backward jeu de taquin slides are concerned. 

Similarly, we will think of $\tau$ as being placed left aligned in a rectangular Ferrers diagram where
\begin{eqnarray*}
\text{number of rows, } n &=& l(\alpha),\\
\text{number of columns} &\geq & \lambda_1+1.
\end{eqnarray*} 
The empty boxes of this Ferrers diagram are assumed to be filled with $0$s. We will assume that the dimensions of the rectangular Ferrers diagram are the same for both $T$ and $\tau$.
The exact dimensions will not matter as long as they satisfy the
constraints mentioned above. We will call the rectangular filling $\tau_{\square}$.

For computing $\tjdt_{i+1}(T)$, we only need to know what the entries in columns $i$ and $i+1$ are. Since we will work with $T_{\square}$ we will take the $0$s into account when considering the $i$-th and $i+1$-th columns. Let the $i$-th and $i+1$-th columns of $T_{\square}$ be as shown below.
\begin{eqnarray*}
\ytableausetup{mathmode,boxsize=1.3em}
\begin{ytableau}
a_1 & b_1\\a_2& b_2\\ \none[\vdots] & \none[\vdots]\\a_n & b_n
\end{ytableau}
\end{eqnarray*}
Furthermore, let $a_{n+1}=b_{n+1}=\infty$. Since we have assumed that there is an addable node in the $i+1$-th column of $\lambda$, we know that 
\begin{align}\label{existenceoutercorner}
\lvert\{b_r: b_r=0, \text{ }1\leq r\leq n\}\rvert >\lvert\{a_r: a_r=0,\text{ }
  1\leq r\leq n\}\rvert.
\end{align}

Let 
\begin{empheq}[box=\mybluebox]{equation*}
min= \text{smallest positive integer such that }a_{min}<b_{min+1}.
\end{empheq}
Then, we have
that 
\begin{empheq}[box=\mybluebox]{equation*}
a_{min}=f_{i+1}(T).
\end{empheq} 
Clearly, $a_{min}>0$. Another trivial observation, that we will use without mentioning is that $a_p=a_q$ for some $1\leq p,q\leq n$ if and only if $a_p=a_q=0$.

Corresponding to $T_{\square}$ above, the $i$-th and $i+1$-th columns of $\tau_{\square}$ are as shown below.
\begin{eqnarray*}
\ytableausetup{mathmode,boxsize=1.3em}
\begin{ytableau}
c_{1} & d_{1}\\c_{2}& d_{2}\\ \none[\vdots] & \none[\vdots]\\ c_{n} & d_{n}
\end{ytableau}
\end{eqnarray*}

Clearly the sequence $c_{1}, c_2,\ldots, c_{n}$ is a rearrangement of the sequence $a_{1}, a_2,\ldots, a_{n}$. A similar statement holds for the sequences $d_{1}, d_2,\ldots, d_{n}$ and $b_{1}, b_2,\ldots, b_{n}$.

Now let us consider what we require from the procedure
$\phi_{i+1}$. We would like it to locate $f_{i+1}(T)$ in
$\tau$ and then remove it from $\tau$ and rearrange
the remaining entries so that what we have is still an $\SSRCT$. The algorithm to
accomplish this will be presented after some intermediate lemmas. For
all the lemmas that follow, we will assume that \textbf{$\mathbf{s_1}$ is the
positive integer such that }
\begin{empheq}[box=\mybluebox]{equation*}
c_{s_1}=a_{min}.
\end{empheq} 

Our first lemma implies that $a_{min}$ does not equal any other entry in the $i$-th and $i+1$-th columns of $T_{\square}$. We know that $a_{min}>0$. Hence it clearly can not equal any other entry in the $i$-th column since $T$ is an $\SSRT$. To see that $a_{min}$ does not equal any entry in the $i+1$-th column, the following suffices.
\begin{Lemma}\label{lemma0}
$a_{min}\neq b_{min}$.
\end{Lemma}
\begin{proof}
Assume, contrary to what is to be established, that $a_{min}=b_{min}$. Since $a_{min}>0$, we get that $min>1$, as
$b_1$ is definitely $0$ given the inequality in \eqref{existenceoutercorner}. Now, since $a_{min}> a_{min-1}$, our assumption gives $a_{min-1}<b_{min}$. But this contradicts the definition of $min$.
\end{proof}
Our next lemma deals with entries greater than $c_{s_1}=a_{min}$ in $\tau_{\square}$.
\begin{Lemma}\label{lemma1}
$c_{t} > c_{s_1} \Longleftrightarrow d_{t} \geq b_{min+1}$.
\end{Lemma}
\begin{proof}
Since $c_{s_1}=a_{min}$, we know that
$b_{min+1}>c_{s_1}$. Thus, $b_{min+1}>c_{j}$ for all
$1\leq j\leq n$ satisfying $c_{j}\leq c_{s_1}$. Thus any entry in the $i+1$-th
column of $T_{\square}$ that is $\geq b_{min+1}$ can not be the neighbour of any entry that is $\leq c_{s_1}$
in the $i$-th column of $\tau_{\square}$. Consider the following sets.
\begin{eqnarray*}
X&=& \{c_{j}: c_{j}> c_{s_1}\} \\
Y&=& \{b_j: b_j\geq b_{min+1}\}
\end{eqnarray*}
The fact that $X$ and $Y$ are sets, and not multisets, follows from the fact that $b_{min+1}>c_{s_1}=a_{min}>0$. Now, we know from our argument at the beginning of this proof that an element of $Y$ can
only 
neighbour an element of $X$ in $\tau_{\square}$. But $\lvert X\rvert=\lvert
Y\rvert$, as $X$ is the same as $\{a_{j}: a_{j}> a_{min}\}$. Hence each element of $Y$ is the neighbour to a unique
element of $X$, and each element of $X$ has a neighbour that is an
element of $Y$. Hence the claim follows.
\end{proof}
We note down an important consequence of the above lemma in the following remark.
\begin{Remark}
Consider the
$i$-th and $i+1$-th columns of $T_{\square}$ again.
\begin{eqnarray*}
\ytableausetup{mathmode,boxsize=3em}
\begin{ytableau}
*(mygreen)a_1 & *(mygreen)b_1\\\none[\vdots] & \none[\vdots]\\*(mygreen)a_{min}&
*(mygreen)b_{min}\\*(orange)a_{min+1}& *(orange)b_{min+1}\\ \none[\vdots] & \none[\vdots]\\*(orange)a_n & *(orange)b_n
\end{ytableau}
\end{eqnarray*}
The boxes in the $i$-th column that contain entries 
$\leq a_{min}$ and the boxes in the $i+1$-th column
that contain entries $\leq b_{min}$ are coloured green. All other
boxes, that is, the boxes in the $i$-th column that contain entries 
$> a_{min}$ and the boxes in the $i+1$-th column
that contain entries $> b_{min}$ are coloured orange. Notice that the sets $X$ and $Y$ defined in the proof of Lemma \ref{lemma1} correspond to the orange entries in the $i$-th column and $i+1$-th column respectively. Then, Lemma
\ref{lemma1} implies that in
$\tau_{\square}$, the values belonging to the green boxes in
$T_{\square}$ get neighbours from the green boxes only, while the
values that belong to the orange boxes in $T_{\square}$ get neighbours
from the orange boxes only.
\end{Remark}

Recall now that $s_1$ is the positive integer such that $c_{s_1}=a_{min}$. Define recursively integers $s_j$ satisfying $1\leq s_j\leq n$ for $j\geq 2$ using the following criteria:
\begin{enumerate}[I.]
\item $s_{j}>s_{j-1}$,
\item $c_{s_j}$ is the greatest positive integer that satisfies
 \begin{empheq}[box=\mybluebox]{equation*} 
 c_{s_{j-1}} > c_{s_j} \geq d_{s_{j-1}}.
 \end{empheq}
\end{enumerate}
Clearly, only finitely many $s_j$ exist. Let the maximum be given by $s_k$ and let 
\begin{empheq}[box=\mybluebox]{equation*} 
S=\{s_1,\ldots ,s_k\}.
\end{empheq}

With this setup, we have the following lemma.
\begin{Lemma}\label{precruciallemma}
If $t>s_j$ for some $1\leq j\leq k$ such that $c_{t}>c_{s_j}$, then $c_{t}>c_{s_1}$.
\end{Lemma}
\begin{proof}
If $j=1$, the claim is clearly true. Hence assume $j\geq 2$. Shown below is the configuration in $\tau_{\square}$ that we will focus on.
\begin{eqnarray*}
\ytableausetup{mathmode,boxsize=2.5em}
\begin{ytableau}
c_{s_{j-1}} & d_{s_{j-1}}\\ \none[\vdots] & \none[\vdots]\\c_{s_{j}}& d_{s_{j}}\\\none[\vdots]\\c_{t}
\end{ytableau}
\end{eqnarray*}
Firstly, note that since $c_{s_1}>c_{s_2}>\cdots
>c_{s_k}$, we can not have $t$ being equal to any $s_i$ for
some $1\leq i\leq k$. For if $t=s_i$ for some $i$, then we know that
$c_{t} < c_{s_j}$ for all $s_j<t$. This contradicts
our assumption on $t$.

We will now show that the hypothesis implies $c_{t}>c_{s_{j-1}}$, as this suffices to establish the claim. We know that
\begin{align}\label{satisfied}
 c_{s_{j-1}}>c_{s_{j}}\geq d_{s_{j-1}}.
\end{align} 
We also know, by the definition of $s_j$, that $s_{j}>s_{j-1}$ is such that $c_{s_j}$ is the greatest positive integer satisfying the inequalities in \eqref{satisfied}. Since $t>s_j$, we know that $c_{t}$ does not satisfy 
\begin{eqnarray}\label{notsatisfied}
c_{s_{j-1}}>c_{t}\geq d_{s_{j-1}}.
\end{eqnarray}  
But clearly $c_{t}$ satisfies $c_{t}>d_{s_{j-1}}$, by using the hypothesis that $c_{t}>c_{s_j}$ and \eqref{satisfied}. Thus, we must have that $c_{t}\geq c_{s_{j-1}}$, if $c_{t}$ is not to satisfy \eqref{notsatisfied}. Now, since $c_{s_{j-1}}$ is positive by definition, we get that so is $c_t$. Thus, in fact, $c_t>c_{s_{j-1}}$.
\end{proof}

\begin{Lemma}\label{precruciallemma2}
Let $t$ be a positive integer such that $s_j<t\neq s_{j+1}$ for some $1\leq j\leq k-1$. Then exactly one of the following holds.
\begin{enumerate}
\item $c_{t}>c_{s_1}.$
\item $c_{t}<c_{s_{j+1}}.$
\end{enumerate}
\end{Lemma}
\begin{proof}
Notice that if $c_{t}>c_{s_j}$, then Lemma \ref{precruciallemma} implies that $c_{t}>c_{s_1}$, which is precisely the first condition in the claim.

Now, assume that $c_{s_j}>c_{t}>c_{s_{j+1}}$. The definition of $s_{j+1}$ implies that $c_{s_j}>c_{s_{j+1}}\geq d_{s_j}$. Thus we get that
\begin{eqnarray}\label{pcl1}
c_{s_j}>c_{t}\geq d_{s_j}.
\end{eqnarray}
But since $c_{t}>c_{s_{j+1}}$ and $t>s_j$, \eqref{pcl1} and the recursive definition of the elements of $S$ would imply that $s_{j+1}$ equals $t$. This is clearly not the case. Therefore $c_{t}\leq c_{s_{j+1}}$. Again using the fact that $c_{s_{j+1}}$ is positive, we get that $c_t<c_{s_{j+1}}$.
\end{proof}

This brings us to the following important lemma which states that $c_{s_k}$ occupies the rightmost box in the bottommost row of size $i$ in $\alpha$, where recall that $\alpha$ is the shape of $\tau$. Here we are identifying $\alpha$ with its reverse composition diagram. Establishing that $c_{s_k}$ occupies the rightmost box in a row of size $i$ is the same as proving that $d_{s_k}=0$. That it belongs to the bottommost row of length $i$ requires us to prove that an entry $c_t$ that lies strictly south of $c_{s_k}$ is either $0$ or is such that its neighbour $d_t>0$. 
\begin{Lemma}\label{cruciallemma}
\begin{enumerate}
\item $d_{s_k}=0$.
\item If $t>s_k$ and $c_{t}>0$, then $c_{t} > c_{s_1}$ and $d_t>0$.
\end{enumerate}
\end{Lemma}
\begin{proof}
Consider first the $i$-th and $i+1$-th columns of $T_{\square}$.
For any positive integer $r$ satisfying $1\leq r\leq min$, consider the following multisets.
\begin{eqnarray*}
X_r&=&\{a_{j}: a_{min}\geq a_{j}\geq b_r\}\\
Y_r&=&\{b_j: b_{min}\geq b_j\geq b_r\}
\end{eqnarray*}
Notice that $X_r$ and $Y_r$ can be multisets only when $b_r=0$. Notice
further that if $b_r=0$, then the cardinalities of the multisets $X_r$
and $Y_r$ are equal. For a generic $r$, the following diagram
describes the elements of $X_r$ (boxes shaded purple in the $i$-th column) and $Y_r$ (boxes
shaded orange in the $i+1$-th column).

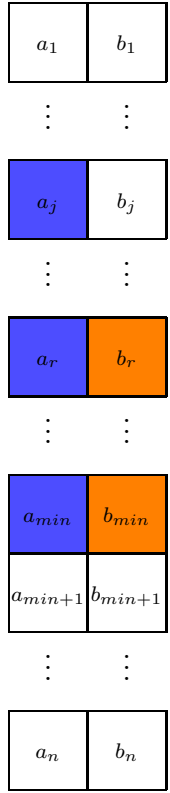
\begin{eqnarray*}
\ytableausetup{mathmode,boxsize=2.5em}
\begin{ytableau}
\scriptstyle a_1 & \scriptstyle \scriptstyle b_1\\\none[\vdots] & \none[\vdots]\\*(anotherblue)\scriptstyle a_{j}&
\scriptstyle b_{j}\\\none[\vdots] & \none[\vdots]\\ *(anotherblue)\scriptstyle a_{r}&
*(orange)\scriptstyle b_{r}\\\none[\vdots] &
\none[\vdots]\\ *(anotherblue)\scriptstyle a_{\scriptscriptstyle min} &
*(orange)\scriptstyle b_{\scriptscriptstyle min}\\\scriptstyle a_{\scriptscriptstyle min+1} & \scriptstyle b_{\scriptscriptstyle min+1}\\ \none[\vdots] & \none[\vdots]\\\scriptstyle a_n & \scriptstyle b_n
\end{ytableau}
\end{eqnarray*}

Assume, contrary to what is to be proved, that $d_{s_k}>0$. Then, we know that there exists a unique positive integer $r$ such that $1\leq r\leq n$ and $d_{s_k} =b_r$, and hence $b_r>0$. Note also that
$c_{s_k}\leq c_{s_1}<b_{min+1}$. This follows since $c_{s_1}=a_{min}$ and $a_{min}<b_{min+1}$.

Now, since
$c_{s_k}\geq d_{s_k}$, we have that
$$b_r=d_{s_k}<b_{min+1}.$$
Using the above, combined with the fact
fact that $b_1$ is always $0$, we get that $1< r\leq min$. Thus, we can consider the multisets $X_r$ and $Y_r$. In fact, they are sets consisting of positive integers as $b_r>0$. We have, by the
definition of $Y_r$, that $$d_{s_k}=b_r\in
Y_r.$$
Furthermore, since
$$a_{min}=c_{s_1}\geq c_{s_k}\geq d_{s_k}=b_r,$$ we
have $c_{s_k}\in X_r$.
Notice now the important fact that the elements of $Y_r$ can only be
neighbours of elements in $X_r$, when we consider their respective positions in the $i$-th and
$i+1$-th columns of $\tau_{\square}$. This follows from Lemma
\ref{lemma1}, in a manner similar to the remark following it.

We know that $a_{r-1}\geq b_r=d_{s_k}$ as $r\leq min$. Hence, we obtain the following inequality.
\begin{eqnarray*}\label{firstcrucialinequality}
|X_r|\geq |Y_r|+1
\end{eqnarray*}
This in turn implies
\begin{eqnarray}\label{secondcrucialinequality}
|X_r\setminus \{c_{s_k}\}|\geq |Y_r\setminus\{d_{s_k}\}|+1.
\end{eqnarray}
Now consider the $i$-th and $i+1$-th columns of
$\tau_{\square}$. Each element of $Y_r$ neighbours some element
of $X_r$. Hence each element of $Y_r\setminus\{d_{s_k}\}$
neighbours a unique element of $X_r\setminus \{c_{s_k}\}$. Now,
\eqref{secondcrucialinequality} implies that there is at least one
element of $X_r\setminus \{c_{s_k}\}$ whose neighbour does not
belong to $Y_r\setminus\{d_{s_k}\}$. This neighbour has
to be strictly smaller than $d_{s_k}$ clearly, as $d_{s_k}=b_r$ is the smallest entry in $Y_r$. The strictness follows from
$d_{s_k}>0$.

If, in $\tau_{\square}$, all elements of the set $X_r\setminus \{c_{s_k}\}$ lie
strictly north of the box occupied by $c_{s_k}$, then the argument earlier implies that for some $c_{j}\in X_r\setminus \{c_{s_k}\}$, we have $d_{j} < d_{s_k}$. But this corresponds to a triple rule violation because of the existence of the configuration shown below.
\begin{eqnarray*}
\ytableausetup{mathmode,boxsize=2em}
\begin{ytableau}
*(anotherblue)c_{j} & d_{j}\\ \none & \none[\vdots]\\\none & *(orange)d_{s_k}
\end{ytableau}
\hspace{3mm} \text{ satisfying }c_{j} \geq d_{s_k}>d_{j}\geq 0 
\end{eqnarray*} 
Hence, there is at least one element $c_j\in X_r\setminus \{c_{s_k}\}$ that lies strictly south of the box occupied by $c_{s_k}$. Any such element has to be greater than $c_{s_k}$. If not then it would satisfy the inequality $c_{s_k}>c_j \geq d_{s_k}$. But this would allows us to define $s_{k+1}$, which we know does not exist. Hence $c_j>c_{s_k}$.
But then Lemma \ref{precruciallemma} implies that $c_j>c_{s_1}$. As $c_{s_1}=a_{min}$, this contradicts the fact that $c_j\in X_r\setminus \{c_{s_k}\}$. Hence, we conclude that $d_{s_k}=0$. This establishes the first part of the claim.

As for the second part, note for any positive integer $j>s_k$ such that $c_{j}>0$, we must have $c_{j}>c_{s_k}$ (again by the fact that $s_{k+1}$ does not exist). Lemma \ref{precruciallemma} implies that $c_{j}>c_{s_1}$, while Lemma \ref{lemma1} implies that $d_{j}\geq b_{min+1}>0$. This establishes the second part of the claim.
\end{proof}

\begin{Remark}The lemma above implies that position of $c_{s_k}$ in
$\tau$ only depends on the shape of $\tau$, and not the filling
itself.
\end{Remark}

We would like to be able to construct the integers $s_{j}$ in
reverse, since locating $c_{s_k}$ is easier than locating
$c_{s_1}$ in $\tau_{\square}$. Thus, we would like an algorithm that takes as input the integer $s_k$ (and $\tau_{\square}$) and terminates by giving $s_1$ as output. Such an algorithm would give us a way of computing $f_{i+1}(T)=a_{min}=c_{s_1}$ given $\tau$. Before we achieve this aim, we need some notation.

Let $q_1$ be the greatest positive integer such that $d_{q_1}=0$ and $c_{q_1}>0$. Define a sequence of positive integers $q_{j}$ satisfying $1\leq q_j \leq n$ for $j\geq 2$ recursively according to the following conditions.
\begin{enumerate}
\item $q_j<q_{j-1}$,
\item $q_{j}$ is the greatest positive integer such that 
 \begin{empheq}[box=\mybluebox]{equation*} 
c_{q_{j}}>c_{q_{j-1}}\geq d_{q_{j}}.
 \end{empheq}
\end{enumerate}
By virtue of Lemma \ref{cruciallemma}, we have that $q_1=s_k$. In fact, $q_{j}=s_{k+1-j}$ for $1\leq j\leq k$ as the next lemma shows.
\begin{Lemma}\label{cruciallemma2}
For $1\leq j\leq k$, we have $q_j=s_{k+1-j}$.
\end{Lemma}
\begin{proof}
Suppose that for some $1\leq j\leq k-1$, we know that $q_j=s_{k+1-j}$. We will show that $q_{j+1}=s_{k-j}$. 
By the definition of $q_{j+1}$, we have 
\begin{eqnarray}\label{qinequality1}
c_{q_{j+1}}>c_{q_j}\geq d_{q_{j+1}}.
\end{eqnarray}
Furthermore, since $q_j=s_{k+1-j}$, we have
\begin{eqnarray}\label{qinequality2}
c_{s_{k-j}}>c_{q_j}\geq d_{s_{k-j}}.
\end{eqnarray}
The recursive definition of $q_{j+1}$ combined with \eqref{qinequality1} and \eqref{qinequality2} implies $q_{j+1}\geq s_{k-j}$. Suppose, contrary to what we seek to prove, that $q_{j+1}>s_{k-j}$. Thus, we have that $q_j>q_{j+1}>s_{k-j}$. Diagrammatically, this corresponds to the following configuration in $\tau_{\square}$.
\begin{eqnarray*}
\ytableausetup{mathmode,boxsize=2.2em}
\begin{ytableau}
c_{s_{k-j}} & d_{s_{k-j}}\\ \none[\vdots] & \none[\vdots]\\c_{q_{j+1}} & d_{q_{j+1}}\\\none[\vdots]\\c_{q_j}
\end{ytableau}
\end{eqnarray*}
From \eqref{qinequality1} and \eqref{qinequality2}, we have $c_{q_{j+1}}>d_{s_{k-j}}$.
Since $q_j=s_{k+1-j}$, we must have $c_{q_{j+1}} > c_{s_{k-j}}$. To see this, notice that had $c_{q_{j+1}}<c_{s_{k-j}}$ been true, the recursive definition of $s_{k+1-j}$ given $s_{k-j}$ would have implied that $s_{k+1-j} \leq q_{j+1}$, which is clearly not the case. Thus, we get that $c_{q_{j+1}}>c_{s_{k-j}}$.

But then, Lemma \ref{precruciallemma} implies $c_{q_{j+1}} > c_{s_1}$. By Lemma \ref{lemma1}, this implies that
\begin{eqnarray}\label{tobecontradicted}
d_{q_{j+1}} \geq b_{min+1}.
\end{eqnarray} 

Since $c_{s_1}\geq c_{s_{k+1-j}}=c_{q_j}$, we
know that $c_{q_j}< b_{min+1}$, as $c_{s_1}=a_{min}<b_{min+1}$. But then the inequality in \eqref{tobecontradicted} is in contradiction to the inequality in \eqref{qinequality1}. Hence $q_{j+1}=s_{k-j}$. Since $q_1=s_k$, the claim follows by induction.
\end{proof}
Thus, now we know that $q_k=s_1$. The next lemma will show that $q_{k+1}$ does not exist.
\begin{Lemma}
There does not exist a positive integer $t$ such that $t<s_1$ and
\begin{eqnarray*}
c_{t} > c_{s_1} \geq d_{t}.
\end{eqnarray*}
\end{Lemma}
\begin{proof}
Suppose such a $t$ exists. By Lemma \ref{lemma1}, we have $d_{t}\geq b_{min+1}$. But we also know that $b_{min+1}>a_{min}=c_{s_1}$. Thus, the inequality in the statement of the lemma can not be satisfied.
\end{proof}

We now know how to compute $c_{s_1}$ given $\tau_{\square}$. Our next aim is to remove $c_{s_1}$ from the $i$-th column of $\tau_{\square}$ and then rearrange entries therein so as to obtain $(\rho^{-1}(\tjdt_{i+1}(T)))_{\square}$. The short algorithm that accomplishes this is presented below.

\begin{Algorithm}\label{singlecolumnadjustment}
\begin{enumerate}
\item For $j=k,k-1,\ldots 2$ in that order, place $c_{q_{j-1}}$ in the box that originally contained $c_{q_j}$.
\item Place a $0$ in the box that originally contained $c_{q_1}$. Call the resulting tableau $\beta(\tau_{\square})$.
\item If there is a row in $\beta(\tau_{\square})$ comprising entirely of $0$s, we remove that row. The resulting rectangular filling is $(\rho^{-1}(\tjdt_{i+1}(T)))_{\square}$. If we omit $0$s altogether, we obtain what we will call $\phi_{i+1}(\tau)$.
\end{enumerate}
\end{Algorithm}
We will demonstrate the above algorithm with an example. The proof of validity of the algorithm will be presented subsequently.
\begin{Example}
Suppose that
\begin{eqnarray*}
T_{\square}=\ytableausetup{mathmode,boxsize=1.25em}
\begin{ytableau}
7&3&0&0&0\\8&5&0&0&0\\9&6&2&1&0
\end{ytableau}.
\end{eqnarray*}
Say we want to compute $\tjdt_3(T)$. It is clear that if we start a backward jeu de taquin slide from the addable node in the third column, then $6$ is the first entry that moves horizontally. Thus $\tjdt_3(T)$ would be obtained by removing $6$ from the second column. Then, we have 
\begin{eqnarray*}
(\emph{tjdt}_3(T))_{\square}=\ytableausetup{mathmode,boxsize=1.25em}
\begin{ytableau}
7&0&0&0&0\\8&3&0&0&0\\9&5&2&1&0
\end{ytableau}.
\end{eqnarray*}
Now we try and repeat the above with the $\SSRCT$ $\tau=\rho^{-1}(T)$ using Algorithm \ref{singlecolumnadjustment}. We have
\begin{eqnarray*}
\tau_{\square}=\ytableausetup{mathmode,boxsize=1.25em}
\begin{ytableau}
7&6&2&1&0\\8&5&0&0&0\\9&3&0&0&0
\end{ytableau}.
\end{eqnarray*}
Then we have $q_1=3$ as $c_{3}$ is the lowermost non-zero entry in the second column that has a neighbour equalling $0$. Since $c_{2}>c_{3}\geq d_{2}$ we have $q_2=2$. Similarly, we have $q_3=1$. Thus, 
\begin{eqnarray*}
\beta(\tau_{\square})=\ytableausetup{mathmode,boxsize=1.25em}
\begin{ytableau}
7&5&2&1&0\\8&3&0&0&0\\9&0&0&0&0
\end{ytableau}.
\end{eqnarray*}
As there is no row comprising entirely of $0$s, by the third step in the algorithm, we have that $(\rho^{-1}(\tjdt_{3}(T)))_{\square}$ is the tableau above. We can remove the $0$s to get a better picture.
Then Algorithm \ref{singlecolumnadjustment} implies that 
\begin{eqnarray*}
\phi_{3}(\tau)=\ytableausetup{mathmode,boxsize=1.25em}
\begin{ytableau}
7&5&2&1\\8&3\\9
\end{ytableau}.
\end{eqnarray*}
We also have
\begin{eqnarray*}
\emph{tjdt}_{3}(T)=\ytableausetup{mathmode,boxsize=1.25em}
\begin{ytableau}
7\\8&3\\9&5&2&1
\end{ytableau}.
\end{eqnarray*}
It is easily checked that $\phi_{3}(\tau) = \rho^{-1}(\tjdt_{3}(T))$.
\end{Example}

Now we will proceed to prove the validity of Algorithm \ref{singlecolumnadjustment} by showing that $\phi_{i+1}(\tau)$ is an $\SSRCT$. Since $\phi_{i+1}(\tau)$ is obtained from $\beta(\tau_{\square})$ by removing extraneous $0$s, we will work with $\beta(\tau_{\square})$ as it is more convenient. We need to show that $\beta(\tau_{\square})$ satisfies the conditions below.
\begin{itemize}
\item Excepting $0$s, the first column of $\beta(\tau_{\square})$ is strictly decreasing from bottom to top.
\item It has no triple rule violations.
\item The rows are weakly decreasing from left to right.
\end{itemize}
The third condition above is verified by an easy argument: Note that if the position of an entry in the $i$-th column remains unchanged, then the condition is automatically verified. Otherwise the entry $c_{s_j}$ is replaced by $c_{s_{j+1}}$ for some $1\leq j\leq k-1$. Since $c_{s_{j+1}}\geq d_{s_j}$ and $c_{s_j}>c_{s_{j+1}}$, we see that the algorithm does not change the fact that the rows were weakly decreasing from left to right. Finally, notice that $c_{s_k}$ is replaced by a $0$, but since $d_{s_k}$ was already $0$ by Lemma \ref{cruciallemma}, we still have rows weakly decreasing from left to right.

Now we will verify that the first condition holds as well. Notice that if $i\geq 2$, then no changes are ever made to the first column by Algorithm \ref{singlecolumnadjustment}. Hence, in this scenario $\phi_{i+1}(\tau)$ has the entries in the first column strictly decreasing from bottom to top. Consider now the case $i=1$. Then, it is easily seen that Algorithm \ref{singlecolumnadjustment} replaces the bottom most non-zero entry (in the first column) whose neighbour is a $0$ with a $0$, and no other change is made. Hence, if one omits the $0$s, the first column of $\beta(\tau_{\square})$ is strictly decreasing from bottom to top.

Thus, we only need to prove that there are no triple rule violations in $\beta(\tau_{\square})$. We will do this in cases.
\begin{Lemma}\label{noinvtripthirdcol}
In $\beta(\tau_{\square})$, consider any configuration of the type shown below
\begin{eqnarray*}
\ytableausetup{textmode}
\begin{ytableau}
x & y\\ \none & \none[\vdots]\\\none & z
\end{ytableau},
\end{eqnarray*}
where $y$ and $z$ belong to the $i+1$-th column. If $z>0$ and $z>y$, then $z>x$.
\end{Lemma}
\begin{proof}
Notice that Algorithm \ref{singlecolumnadjustment} does not alter the position of any entry in the $i+1$-th column. Let $x'$ be the entry in $\tau_{\square}$ in the box currently occupied by $x$ in $\beta(\tau_{\square})$. Since there are no triple rule violations in $\tau_{\square}$, we know that
\begin{eqnarray}\label{nononinv}
 z>0 \text{ and } z>y \Longrightarrow z>x'
\end{eqnarray}
We have two possibilities to deal with. Firstly, if $x=x'$, then the claim follows easily from \eqref{nononinv}. Now, assume $x\neq x'$. Algorithm \ref{singlecolumnadjustment} implies that $x<x'$. Again, \eqref{nononinv} implies the claim.
\end{proof}

\begin{Lemma}\label{noinvtripsecondcol}
In $\beta(\tau_{\square})$, consider any configuration of the type shown below
\begin{eqnarray*}
\ytableausetup{textmode}
\begin{ytableau}
x & y\\ \none & \none[\vdots]\\\none & z
\end{ytableau},
\end{eqnarray*}
where $y$ and $z$ belong to the $i$-th column. If $z>0$ and $z>y$, then $z>x$.
\end{Lemma}
\begin{proof}
Firstly, notice that we only need to worry about this configuration when $i\geq 2$.

Let $x'$, $y'$ and $z'$ be the entries in $\tau_{\square}$ in the boxes corresponding to those occupied by $x$, $y$ and $z$ in $\beta(\tau_{\square})$ respectively. Since Algorithm \ref{singlecolumnadjustment} does not change the $i-1$-th column, we have $x=x'$.
This along with the fact that there are no triple rule violations in $\tau_{\square}$ implies 
\begin{eqnarray*}\label{nononinvsecond}
 z'>0 \text{ and } z'>y' \Longrightarrow z'>x.
\end{eqnarray*}
If $y=y'$ and $z=z'$, then we have nothing to prove.

Now assume that $z=z'$ but $y\neq y'$. Then we have two cases. In the first case, we have $y'=c_{s_k}$ and thus, $y=0$. Furthermore, since $z$ occupies a box below $c_{s_k}$, we know that $z>c_{s_1}$, by Lemma \ref{cruciallemma}. Now the claim follows readily for this case. In the second case, we have $y'=c_{s_j}$ and thus, $y=c_{s_{j+1}}$ for some $1\leq j\leq k-1$. The possible relative positions of $y$, $y'$ and $z$ in $\tau_{\square}$ are shown below.
\begin{eqnarray*}
\ytableausetup{mathmode}
\begin{ytableau}
y' \\  \none[\vdots]\\y\\ \none[\vdots]\\ z
\end{ytableau}
\hspace{4mm}\hspace{7mm}
\ytableausetup{mathmode}
\begin{ytableau}
y' \\  \none[\vdots]\\z\\ \none[\vdots]\\ y
\end{ytableau}
\end{eqnarray*}
According to Lemma \ref{precruciallemma2}, we have either $z>c_{s_1}$ (and hence $z>y,y'$) or $z<y$ (and hence $z<y'$). In any case, we see that if $z>0$ and $z>y$, then $z>x$.

Now assume that $y\neq y'$ and $z\neq z'$. Then we must have $y=c_{s_{i_1}}$ and $z=c_{s_{i_2}}$ where $2\leq i_1<i_2\leq k$. But then the situation where $z>y$ does not arise.
\end{proof}
Thus, we have established that $\phi_{i+1}(\tau)$ is an $\SSRCT$ indeed. We will prove another lemma that is of similar flavour to the previous two lemmas, and which we will soon use.
\begin{Lemma}\label{lemmaforlastrow}
In $\beta(\tau_{\square})$, consider a configuration of the type 
\begin{eqnarray*}
\ytableausetup{mathmode}
\begin{ytableau}
x & y
\end{ytableau}
\hspace{5mm}\text{ where $x$ is in $i$-th column and $y$ in $i+1$-th column.}
\end{eqnarray*}
If $c_{s_1}>y\geq 0$ then $c_{s_1}>x$.
\end{Lemma}
\begin{proof}
For any configuration of the type in the statement of the claim, we first show that $c_{s_1}$ can neither equal $x$ nor $y$.

Recall that $c_{s_1}=a_{min}$ is the first entry that moves horizontally when computing $\jdt_{i+1}(T)$. This entry is removed from the $i$-th column of $\tau_{\square}$ when $\beta(\tau_{\square})$ is computed. Hence it can not equal $x$. Lemma \ref{lemma0} implies it can not equal $y$ either.

If $c_{s_1}>y\geq 0$, then we know that $y\leq b_{min}$. But then Lemma \ref{lemma1} yields that $x<c_{s_1}$, which is what we wanted.
\end{proof}
We are now in a position to state our first important theorem.
\begin{Theorem}\label{commutativediagramforPhi}
The following diagram commutes.
\begin{eqnarray*}
\xymatrix{ \text{$\SSRT$\emph{s}} \ar[d]_{\tjdt_{i+1}} \ar[r]^{\rho^{-1}} & \text{$\SSRCT$\emph{s}}\ar[d]^{\phi_{i+1}}\\ \text{$\SSRT$\emph{s}} \ar[r]_{\rho^{-1}} & \text{$\SSRCT$\emph{s}}}
\end{eqnarray*}
\end{Theorem}
\begin{proof}
We have already established that $\phi_{i+1}(\tau)$ is an $\SSRCT$. This observation and the fact that we have removed the same entry from the $i$-th column in both $T$ and $\tau$ while ensuring that none of the other entries have changed columns, establishes the claim.
\end{proof}

We define another $\SSRCT$ using the procedure $\phi_i$ as follows.
\begin{eqnarray*}
\Phi_{i}(\tau) &=& \left\lbrace\begin{array}{ll} \phi_{2}\circ \phi_{3}\circ \cdots \circ \phi_{i}(\tau) & i>1\\ \tau & i=1\end{array}\right.
\end{eqnarray*}
Let $\gamma = \sha(\Phi_{i}(\tau))=(\gamma_1,\ldots, \gamma_s)$. Let $\gamma'=(\gamma,i)$ be the composition obtained by appending a part of length $i$ to $\gamma$. Consider the filling of shape $\gamma' \cskew (1)$ where the first $s$ rows are filled as $\Phi_{i}(\tau)$. Fill the last row with the entries of $H_{i}(T)$ in decreasing order from left to right. Call the resulting filling $\mu_{i}(\tau)$. 

\begin{Theorem}\label{backwardjdtslideonssrct}
With the notation as above, the following diagram commutes.
\begin{eqnarray*}
\xymatrix{ \text{$\SSRT$\emph{s}} \ar[d]_{\jdt_{i}} \ar[r]^{\rho^{-1}} & \text{$\SSRCT$\emph{s}}\ar[d]^{\mu_{i}}\\ \text{$\SSRT$\emph{s}} \ar[r]_{\rho_{(1)}^{-1}} & \text{$\SSRCT$\emph{s}}}
\end{eqnarray*}
\end{Theorem}
\begin{proof}
First, we start by showing that $\mu_{i}(\tau)$ is actually an $\SSRCT$. Notice that the entries in every column are all distinct, using Lemma \ref{lemma0}. Next, since the rows are weakly decreasing by construction, and the top $s$ rows already form an $\SSRCT$, we just need to check that there is no configuration of the form
\begin{eqnarray*}
\ytableausetup{mathmode,boxsize=1.3em}
\begin{ytableau}
x & y\\ \none &\none[\vdots]\\\none & z
\end{ytableau}
\hspace{5mm}\text{ satisfying } x\geq z>y \text{ and $z$ is in the bottommost row.}
\end{eqnarray*}
But this follows from Lemma \ref{lemmaforlastrow}. Hence $\mu_{i}(\tau)$ is an $\SSRCT$. Now the claim follows as the set of entries in each column of $\jdt_{i}(T)$ and $\mu_{i}(\tau)$ is the same.
\end{proof}
Given the theorem above, we are justified in saying that $\mu$ is an analogue of the backward jeu de taquin slide. From this point on, by a backward jeu de taquin slide on $\SSRCT$s of straight shape we will mean an application of $\mu$. 

\subsection{Skew shape after $\mu$}\label{section: jdt operators}
If we perform backward jeu de taquin slides on $\SSRT$s of straight shape, it is easy to predict the resulting shape once the slide is completed. In the case of $\SSRCT$s, when we apply the procedure $\mu$, it is not immediate what the resulting shape is. The aim of this subsection is to use the operators on compositions defined in Section \ref{section: operators on compositions} to predict it.
\begin{Lemma}\label{shapephi}
The shape of $\phi_{i+1}(\ctau)$ is $\down_i(\alpha)$ if $i\geq 1$.
\end{Lemma}
\begin{proof}
It is clear from Algorithm \ref{singlecolumnadjustment} that box that contained $c_{s_k}$ in $\ctau_{\square}$ contains $0$ in $\phi_{i+1}(\ctau_{\square})$. We also know from Lemma \ref{cruciallemma} that all parts of $\alpha$ that are below $\alpha_{s_k}$ in the reverse composition diagram corresponding to $\alpha$ have either length strictly less than $\alpha_{s_k}$ or strictly greater than $\alpha_{s_k}$. Furthermore, the same lemma also implies that $\alpha_{s_k}$ has length $i$. Thus, the claim follows.
\end{proof}
As an immediate corollary, we obtain the following.
\begin{Corollary}\label{shapePhi}
The shape of $\Phi_{i+1}(\ctau)$ is $\downrow_{[i]}(\alpha)$ if $i\geq 0$.
\end{Corollary}
\begin{proof}
$\Phi_{1}$ fixes the $\SSRCT$ $\ctau$. In this case the claim is clear. If we are computing $\Phi_{i+1}(\ctau)$ for $i\geq 1$, we are computing $\phi_{2}\circ \cdots \circ \phi_{i+1}(\ctau)$. Using Lemma \ref{shapephi} repeatedly, we get that $\Phi_{i+1}(\ctau)$ has shape $\down_{1}\cdots \down_{i}(\alpha)=\downrow_{[i]}(\alpha)$.
\end{proof}
\begin{Theorem}\label{shapemu}
The $\SSRCT$ $\mu_{i}(\ctau)$ is of shape $\gamma\cskew (1)$ where $\gamma=\up_i(\alpha)$.
\end{Theorem}
\begin{proof}
This follows from the description of the procedure $\mu_{i}$ and Corollary \ref{shapePhi}.
\end{proof}

\section{Backward jeu de taquin slides for SSRCTs of skew shape}\label{section: slides for skew shape}
After giving a description for the backward jeu de taquin slide for $\SSRCT$s of straight shape, it is only natural to ask for a generalization, that is, a backward jeu de taquin slide for $\SSRCT$s of other shapes. For the proofs in this section, we will consider $\SSRCT$s whose entries are drawn from the ordered alphabet $X=A\cup \tcr{\bar{A}}$ where
\begin{eqnarray*}
A&=&\{1<2<3<\cdots\}\\
\tcr{\tcr{\tcr{\bar{A}}}}&=&\{\tcr{\bar{1}}<\tcr{\bar{2}}<\tcr{\bar{3}}<\cdots\}.
\end{eqnarray*}
We will further assume that $\tcr{\bar{1}}$ is greater than any letter in $A$. Thus, we have a total order on $X$.

Consider a skew reverse composition shape $\alpha \cskew\beta$. Let $\ctau_o$ be an $\SSRCT$ of shape $\alpha\cskew\beta$ constructed using the alphabet $A$, and let $\ctaup$ be an $\SSRCT$ of straight shape $\beta$ constructed using the alphabet $\tcr{\bar{A}}$. Then we will denote by $\tau$ the $\SSRCT$ of straight shape $\alpha$ formed by filling the inner shape in $\ctau_o$ according to $\ctaup$.

Let $i\geq 1$. Assume, for the moment, that we are computing $\phi_{i+1}(\tau)$. We will show that the shape of the $\SSRCT$ obtained by considering entries that belong to $A$ in $\phi_{i+1}(\tau)$ is independent of the filling $\ctaup$. The example next will clarify what we are aspiring to.
\begin{Example}
Consider the two tableaux below.
\begin{eqnarray*}
\ctau_o=\ytableausetup{centertableaux}
\begin{ytableau}
3\\\bullet & \bullet & 2& 1\\\bullet & \bullet & \bullet & 5\\\bullet & 4
\end{ytableau}
\text{ and }
\ctaup=\ytableausetup{centertableaux}
\begin{ytableau}
\tcr{\bar{3}}& \tcr{\bar{1}}\\\tcr{\bar{5}} & \tcr{\bar{4}} & \tcr{\bar{2}}\\\tcr{\bar{6}}
\end{ytableau}
\end{eqnarray*}
Thus
\begin{eqnarray*}
\tau=\ytableausetup{centertableaux}
\begin{ytableau}
3\\\tcr{\bar{3}} & \tcr{\bar{1}} & 2& 1\\\tcr{\bar{5}} &\tcr{\bar{4}}& \tcr{\bar{2}} & 5\\\tcr{\bar{6}} & 4 
\end{ytableau}.
\end{eqnarray*}
Then $\phi_{3}(\tau)$ is the $\SSRCT$ below.
\begin{eqnarray*}
\ytableausetup{mathmode,boxsize=1.3em}
\begin{ytableau}
3\\\tcr{\bar{3}} & 4 & 2& 1\\\tcr{\bar{5}} &\tcr{\bar{4}}& \tcr{\bar{2}} & 5\\\tcr{\bar{6}}  
\end{ytableau}
\end{eqnarray*}
If we replace the entries in the above tableau that belong to $\tcr{\bar{A}}$ by bullets, then we obtain the following $\SSRCT$.
\begin{eqnarray*}
\ytableausetup{mathmode,boxsize=1.3em}
\begin{ytableau}
3\\\bullet & 4 & 2& 1\\\bullet &\bullet & \bullet & 5\\\bullet  
\end{ytableau}
\end{eqnarray*}
We would like to demonstrate that the above tableau is obtained independent of our choice of the filling $\ctaup$.
\end{Example}
We will use the same notation as the one we established when we described the process $\phi_{i+1}$. Thus, our focus is on the $i$-th and $i+1$-th columns of $T_{\square}$ and $\tau_{\square}$ respectively, as shown below. Recall that $\rho(\tau)=T$.
\begin{eqnarray*}
\ytableausetup{mathmode,boxsize=1.3em}
\begin{ytableau}
a_1 & b_1\\a_2& b_2\\ \none[\vdots] & \none[\vdots]\\a_n & b_n
\end{ytableau}
\hspace{20mm}
\begin{ytableau}
c_{1}& d_{1}\\c_{2}& d_{2}\\ \none[\vdots] & \none[\vdots]\\ c_{n} & d_{n}
\end{ytableau}
\end{eqnarray*}
The only difference from the earlier situation is that the entries are now allowed to be elements of $\tcr{\bar{A}}$.
We will recall the notation here for the sake of convenience. We have that $a_{min}$ is the first entry that moves horizontally when a backward jeu de taquin slide is initiated from the $i+1$-th column in $T$. Hence $min$ is the smallest positive integer such that $a_{min}<b_{min+1}$. Also we assume that, in $\tau_{\square}$, the entry corresponding to $a_{min}$ is $c_{s_1}$, that is, $s_1$ is the positive integer such that $c_{s_1}=a_{min}$. Finally, we define recursively integers $s_j$ satisfying $1\leq s_j\leq n$ for $j\geq 2$ using the following criteria.
\begin{enumerate}
\item $s_{j}>s_{j-1}$,
\item $c_{s_j}$ is the greatest positive integer that satisfies
 \begin{empheq}[box=\mybluebox]{equation} 
 c_{s_{j-1}} > c_{s_j} \geq d_{s_{j-1}}.
 \end{empheq}
\end{enumerate}
Clearly, only finitely many $s_j$ exist. Let the maximum be given by $s_k$ and let 
\begin{empheq}[box=\mybluebox]{equation} 
S=\{s_1,\ldots ,s_k\}.
\end{empheq}

\begin{Lemma}\label{cruciallemmaforskewshape}
Assume $c_{s_1}\in \tcr{\bar{A}}$. Let $1\leq j\leq k$ be the greatest positive integer such that $c_{s_j}$ belongs to $\tcr{\bar{A}}$. If $t>s_j$ is a positive integer such that $c_{t}>0$, then one of the following statements holds.
\begin{enumerate}
\item $c_{t}\in A$.
\item $c_{t}\in\tcr{\bar{A}}$ and $d_{t}\in \tcr{\bar{A}}$.
\end{enumerate}
\end{Lemma}
\begin{proof}
Firstly, note that $c_{s_1}\in \tcr{\bar{A}}$ implies that $b_{min+1}\in
\tcr{\bar{A}}$. Notice also that if $c_{s_j} \in \tcr{\bar{A}}$, then for all $1\leq
r\leq j$, we must have that $c_{s_r}\in \tcr{\bar{A}}$ as
$c_{s_1}>\cdots > c_{s_j}>\cdots > c_{s_k}$.

If $c_{t}\in A$, then the first condition is already satisfied. 
Hence assume that $c_{t}\in \tcr{\bar{A}}$. 
If $j=k$, then Lemma \ref{cruciallemma} implies that $c_{t}>c_{s_1}$. 
Then, Lemma \ref{lemma1} implies $d_{t} \geq
b_{min+1}$. Hence $d_{t} \in \tcr{\bar{A}}$.

Now assume that $j<k$. We have that $c_{s_j}\in \tcr{\bar{A}}$ and
$c_{s_{j+1}}\in A$. Since the definition of $s_{j+1}$ implies that
\begin{eqnarray*}
c_{s_{j}}>c_{s_{j+1}}\geq d_{s_j},
\end{eqnarray*} 
we get that $d_{s_j}\in A$. 

By our assumptions, we know that $t>s_j$ and
$c_{t}\in \tcr{\bar{A}}$. Using this and the fact that $j$ is the largest integer such that $c_{s_j}\in \tcr{\bar{A}}$, we conclude that $t\neq s_{j+1}$. But then, we must have that
$c_{t}>c_{s_j}$. To see this, note that if $c_t<c_{s_j}$, then we have the following inequality using the fact that $d_{s_j}\in A$.
\begin{align*}
c_{s_j}>c_t>d_{s_j}
\end{align*}
But then, we get that $c_{s_{j+1}}>c_t$ and hence, must belong to $\tcr{\bar{A}}$, which is not the case. Hence $c_t>c_{s_j}$ indeed.

Lemma \ref{precruciallemma2}
implies the inequality $c_{t}>c_{s_1}$ and then Lemma \ref{lemma1} implies that $d_{t} \geq
b_{min+1}$. Thus, we get that $d_{t} \in \tcr{\bar{A}}$.
\end{proof}

Note that the above lemma assumes nothing about the filling $\ctaup$.
\begin{Lemma}\label{lemma: innershapechange}
Assume $c_{s_1}\in \tcr{\bar{A}}$. Consider the $\SSRCT$ of straight shape
inside $\phi_{i+1}(\tau)$ formed by the entries that belong to $\tcr{\bar{A}}$. The shape of this $\SSRCT$ is $\down_{i}(\beta)$.
\end{Lemma}
\begin{proof}
Let $1\leq j\leq k$ be the greatest positive integer such that $c_{s_j}\in \tcr{\bar{A}}$. Then from Lemma \ref{cruciallemmaforskewshape}, it follows that $c_{s_j}$ occupies the rightmost box in the bottommost row of length $i$ in the reverse composition diagram of $\beta$. According to Algorithm \ref{singlecolumnadjustment}, this entry gets replaced either by an element of $A$ or $0$. In either case, the shape of the $\SSRCT$ formed by the entries that belong to $\tcr{\bar{A}}$ is $\down_{i}(\beta)$.
\end{proof}

\begin{Lemma}\label{lemma: exiting entry from A}
Assume that $c_{s_1}\in A$. Then the $\SSRCT$ in $\phi_{i+1}(\tau)$ formed by considering the entries that belong to $A$ is independent of the choice of $\bar{\tau}$. Furthermore, this $\SSRCT$ has shape $\down_{i}(\alpha) \cskew \beta$.
\end{Lemma}
\begin{proof}
Since $c_{s_1}$ belongs to $A$ and $c_{s_1} > c_{s_j}$ for all $2\leq j\leq k$, we get that $c_{s_j}\in A$ for $2\leq j\leq k$. Note by Lemma \ref{cruciallemma} that $c_{s_k}$ occupies the rightmost box in the bottommost part of length $i$ in the reverse composition diagram of $\alpha$. Now Algorithm \ref{singlecolumnadjustment} immediately implies the claim.
\end{proof}

\begin{Lemma}\label{lemma: exiting entry from Abar}
Assume that $c_{s_1}\in \tcr{\bar{A}}$. Then the $\SSRCT$ in $\phi_{i+1}(\tau)$ formed by considering the entries that belong to $A$ is independent of the choice of $\bar{\tau}$. Furthermore, this $\SSRCT$ has shape $\down_{i}(\alpha) \cskew \down_{i}(\beta)$.
\end{Lemma}
\begin{proof}
Let $j$ be the greatest positive integer satisfying $1\leq j\leq k$ such that $c_{s_j}\in \tcr{\bar{A}}$. By Lemma \ref{lemma: innershapechange}, we know that $c_{s_j}$ occupies the rightmost box in the bottommost row of length $i$ in the reverse composition diagram of $\beta$. Also, by Lemma \ref{cruciallemma} we know that $c_{s_k}$ occupies the rightmost box in the bottommost row of length $i$ in the reverse composition diagram of $\alpha$ (recall also that we draw the reverse composition diagram of $\beta$ in the bottom left corner of the reverse composition diagram of $\alpha$ when depicting $\alpha\cskew \beta$). Observe now that $s_{j}$ does not depend on $\bar{\tau}$. Thus, from Algorithm \ref{singlecolumnadjustment}, it follows that the $\SSRCT$ in $\phi_{i+1}(\tau)$ formed by the entries that belong to $A$ is independent of $\bar{\tau}$ and that it has shape $\down_{i}(\alpha) \cskew \down_{i}(\beta)$.
\end{proof}
Now we define $\mu_i(\tau_o)$ as follows.
\begin{align*}
\mu_i(\tau_o)= \text{ the $\SSRCT$ formed by considering the entries that belong to $A$ in $\mu_i(\tau)$}.
\end{align*}
It is not clear from the above definition that $\mu_i(\tau_o)$ does not depend on the choice of $\bar{\tau}$. We establish this next.

Let $T_o=\rho_{\beta}(\ctau_o)$ and $\bar{T}=\rho(\ctaup)$. Define $T$ to be $\rho(\tau)$. Consider a backward jeu de taquin slide on $T$ starting from an addable node in column $i$. Let $j$ be the greatest integer satisfying $1\leq j\leq i-1$ such that entry that moves horizontally from column $j$ to column $j+1$ when computing $\jdt_{i}(T)$ is an element of $\tcr{\bar{A}}$. If there is no such $j$, define $j$ to be $0$. If $j\geq 1$, then the entries moving horizontally when computing $\jdt_{i}(T)$ from all columns between 1 and $j$ (inclusive) are all elements of $\tcr{\bar{A}}$, because the entries moving horizontally during a backward jeu de taquin slide increase weakly as we move from column $i-1$ to column $1$. Now by using Theorem \ref{commutativediagramforPhi}, we get that if we compute $\phi_{2}\circ \cdots \circ\phi_{i}(\tau)$, all the entries that exit from columns between $1$ and $j$ belong to $\tcr{\bar{A}}$.

The preceding discussion along with Lemmas \ref{lemma: exiting entry from A} and \ref{lemma: exiting entry from Abar} implies the following proposition.
\begin{Proposition}\label{innershapechange2}
Let $\gamma$ be the $\SSRCT$ $\Phi_{i}(\ctau)$ where $i\geq 1$. Then the $\SSRCT$ formed by considering the entries in $\gamma$ that belong to $\tcr{\bar{A}}$ has shape $\downrow_{[j]}(\beta)$. The $\SSRCT$ formed by considering the entries in $\gamma$ that belong to $A$ has shape $\downrow_{[i-1]}(\alpha)\cskew \downrow_{[j]}(\beta)$. Furthermore, this $\SSRCT$ is independent of the choice of $\bar{\tau}$.
\end{Proposition}
This proposition along with our description of the procedure $\mu$ prior to Theorem \ref{backwardjdtslideonssrct}, implies the following theorem.
\begin{Theorem}\label{theorem: backward slides for skew shape}
Let $\gamma$ be the $\SSRCT$ $\mu_{i}(\ctau)$ where $i\geq 1$. Then the $\SSRCT$ formed by considering the entries in $\gamma$ that belong to $\tcr{\bar{A}}$ has shape $\up_{j+1}(\beta)\cskew(1)$. The $\SSRCT$ formed by considering the entries in $\gamma$ that belong to $A$ has shape $\up_{i}(\alpha)\cskew\up_{j+1}(\beta)$, and is independent of the choice of $\bar{\tau}$.
\end{Theorem}
The theorem above proves that $\mu_{i}(\tau_o)$ is well-defined and does not depend on the choice of filling $\bar{\tau}$. From this point on, we will not require the alphabet $\tcr{\bar{A}}$.

\section{A new poset on compositions}\label{section: new poset}
In this section we will introduce a new poset structure on the set of compositions that arises from performing backward jeu de taquin slides on $\SSRCT$s. But before that we need to understand the link between maximal chains in Young's lattice $\mathcal{Y}$ and $\SRT$s. 

\subsection{Growth words and $\SRT$s}
With any maximal chain in $\mathcal{Y}$ starting at the empty partition, we can associate a sequence of positive integers as explained next. Let $$\lambda_0=\varnothing \prec \lambda_1\prec \cdots \prec \lambda_n$$ be a maximal chain. Then each partition $\lambda_{k}$ for $1\leq k\leq n$ is obtained by adding a box at an addable node of $\lambda_{k-1}$, say, in column $i_k$. The sequence $i_n\cdots i_1$ is called the \textit{column growth word} corresponding to this maximal chain. Note that the column growth word is always \textit{reverse lattice}, that is, in any suffix of the column growth word the number of $j$s weakly exceeds the number of $j+1$s for all positive integers $j$. What we have outlined is clearly a bijection between reverse lattice words of length $n$ and maximal chains in $\mathcal{Y}$ starting at $\varnothing$ and ending at a partition of $n$. We will now associate an $\SRT$ of size $n$ given a column growth word $w=i_n\cdots i_1$. This is done by considering the corresponding maximal chain and assigning $n-k+1$ to the box in $\lambda_{k}$ that is not in $\lambda_{k-1}$. This $\SRT$ will be denoted by $T_w$.

Before we state our next lemma connecting the column reading word of $T_w$ to its associated reverse lattice word, we require the notion of standardization. Given a word $w=i_n\cdots i_1$ where the $i_j$s are positive integers, define the \textit{inversion set} as follows.
\begin{align*}
\inv  (w)=\{(p,q): i_p<i_q\}
\end{align*}
There exists a unique permutation in $\sigma\in\mathfrak{S}_n$ such that $\inv (\sigma)=\inv (w)$. We call $\sigma$ the \textit{standardization} of $w$, and denote it by $\std(w)$.

\begin{Lemma}\label{lemma: reading word is inverse of standardized growth word}
Let $w=i_n\cdots i_1$ be a reverse lattice word, and let $\sigma = \std(w)$. Then the column reading word of $T_{w}$ equals $\sigma^{-1}$.
\end{Lemma}
\begin{proof}
Notice that all the instances of a positive integer $j$ in $w$ are in positions given by the entries in the $j$-th column of $T_w$ read in increasing order. This implies the claim.
\end{proof}
\begin{Example}
Consider the reverse lattice word $w=\tc{mygreen}{3}4\tcr{11}\tcb{2}\tc{mygreen}{3}\tcr{1}\tcb{2}\tcr{1}$. Then we have that $\sigma=\std(w)=\tc{mygreen}{7}9\tcr{1}\tcr{2}\tcb{5}\tc{mygreen}{8}\tcr{3}\tcb{6}\tcr{4}$ in single line notation. Notice that $T_w$ is the $\SSRT$ below.
\begin{eqnarray*}
\ytableausetup{mathmode, boxsize=1.3em}
\begin{ytableau}
\tcr{3}\\ \tcr{4}\\
\tcr{7}&\tcb{5} & \tc{mygreen}{1}\\
\tcr{9} & \tcb{8} &\tc{mygreen}{6} & 2
\end{ytableau}
\end{eqnarray*}
The column reading word of $T_w$ is the following permutation.
\begin{eqnarray*}
\tcr{3479}\tcb{58}\tc{mygreen}{16}2
\end{eqnarray*}
It can be checked that the permutation above is precisely $\sigma^{-1}$.
\end{Example}
Now recall that if we perform our variant of the Robinson-Schensted algorithm on the column reading word of a tableau $T$, we recovers $T$ as the insertion tableau. This implies the following corollary to the previous lemma, establishing $T_w$ as a insertion tableau for permutation associated naturally to it.
\begin{Corollary}\label{corollary: insertion tableau and reading word}
Let $w=i_n\cdots i_1$ be a reverse lattice word, and let $\sigma=\std(w)$. Then $T_w=P(\sigma^{-1})$.
\end{Corollary}
This interepretation of $T_w$ will be of importance to us in the next subsection.

\subsection{A poset on compositions}
We will now define a new poset on the set of compositions, denoted by $\mathcal{R}_{c}$, using its cover relation $\lessdot_{r}$ defined by setting $\alpha \lessdot _{r} \beta$ if and only if $\beta=\up_i(\alpha)$ for some $i\geq 1$. The order relation $<_{r}$ in $\mathcal{R}_{c}$ is obtained by taking the transitive closure of the cover relation $\lessdot_{r}$.

\begin{Example}\label{example: cover relation right Pieri poset}
Let $\alpha= (3,1,4,2,1)$ and $\beta=u_{4}(\alpha)=(2,1,4,1,4)$. Then $\alpha \lessdot_{r} \beta$ in $\mathcal{R}_c$.
\end{Example}
The next theorem is about counting maximal chains in $\mathcal{R}_{c}$, and the reader should compare it with the analogous result in the case of Young's lattice $\mathcal{Y}$ \cite[Proposition 7.10.3]{stanley-ec2} and, more recently, in the case of $\mathcal{L}_c$ \cite[Proposition 2.11]{BLvW}.
\begin{Theorem}\label{theorem: counting maximal chains in right pieri poset}
The number of maximal chains from $\varnothing$ to a composition $\alpha$ in $\mathcal{R}_{c}$ is equal to the number of $\SRCT$s of shape $\alpha$. 
\end{Theorem}
To prove the above theorem, we will use our variant of the Robinson-Schensted algorithm defined in Subsection \ref{subsection: vrsk} and answer a more general question. We will start by proving a result that allows us to compute the resulting inner shape after a sequence of backward jeu de taquin slides has been applied to an $\SSRCT$ of straight shape.

Let $\tau$ be an $\SSRCT$ of shape $\alpha$. Consider the computation of $\mu_{i_n}\cdots\mu_{i_1}(\tau)$, where we assume that the sequence of backward jeu de taquin slides we are executing is valid at all steps. This means that for all $1\leq k\leq n-1$, the outer shape $\beta$ underlying the $\SSRCT$ $\mu_{i_k}\cdots \mu_{i_1}(\tau)$ is such that $\widetilde{\beta}$ has an addable node in column $i_{k+1}$. Furthermore, we also assume that $\widetilde{\alpha}$ has an addable node in column $i_1$, so that we can compute $\mu_{i_1}(\tau)$ to start with.

We will deviate a little from our earlier description of backward jeu de taquin slides for an $\SSRCT$ of straight shape. Let $max(\tau)$ denote the maximum entry in the $\SSRCT$ $\tau$ (if $\tau=\varnothing$ then this is $0$). We will assume that in computing $\mu_{i}(\tau)$, instead of getting a $\SSRCT$ of skew reverse composition shape, with the vacant lower left corner filled by a bullet $\bullet$, we insert the number $max(\tau)+1$ in the corner left vacant. In this way we ensure that $\mu_i(\tau)$ is also an $\SSRCT$ of straight shape. We will do backward jeu de taquin slides (the procedure $\jdt$) on $\SSRT$s of straight shape with a similar rule.
\begin{Example}
Consider the $\SRCT$ $\tau=$
\ytableausetup{mathmode,boxsize=1em}
\begin{ytableau}
5&2\\8&7&4&1\\9&6&3\\11&10
\end{ytableau}
of shape $(2,4,3,2)$. We will compute $\mu_{4}\mu_{3}\mu_{4}(\tau)$ under the set of rules outlined prior to the example. Firstly, $max(\tau)=11$. We obtain the following $\SSRCT$s sequentially (where the entries that come in the lower left corner are highlighted).
$$\ytableausetup{mathmode,boxsize=1.3em}
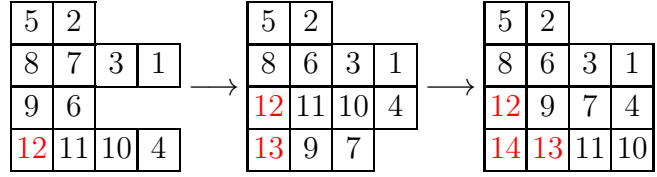
\begin{ytableau}
5&2\\8&7&3&1\\9&6\\\tcr{12} & 11 &  10 & 4
\end{ytableau}\longrightarrow
\begin{ytableau}
5&2\\8&6&3&1\\ \tcr{12} &11 &10 &4\\ \tcr{13} & 9 &7
\end{ytableau}\longrightarrow
\begin{ytableau}
5&2\\8&6&3&1\\\tcr{12} & 9 & 7 & 4\\\tcr{14} &\tcr{13} & 11 & 10
\end{ytableau}$$
\end{Example}
We will next state a theorem that allows us to compute the final $\SRCT$ (under our new setup) obtained when a sequence of slides is applied to an empty $\SRCT$ $\varnothing$.
\begin{Theorem}\label{theorem: jdt sequences and recording tableau}
Let $w=i_n\cdots i_1$ be a reverse lattice word, and let $\sigma=\std(w)$. Then 
\begin{align*}
\mu_{i_n}\cdots \mu_{i_1}(\varnothing)=\rho^{-1}(Q(\sigma)).
\end{align*}
\end{Theorem} 
\begin{proof}
Let $\mu_{i_n}\cdots \mu_{i_1}(\varnothing)=\tau$. Clearly, $\tau$ is an $\SRCT$. Now consider the corresponding sequence of slides starting from the empty $\SRT$, also denoted $\varnothing$, and suppose that the final $\SRT$ obtained is $T$. Thus, $\jdt_{i_n}\cdots \jdt_{i_1}(\varnothing)=T$. From Theorem \ref{backwardjdtslideonssrct}, it follows that $\rho(\tau)=T$ in this new setting. 

Observe that in obtaining $T$, we encounter a sequence of $\SRT$s for $1\leq k \leq n$ as follows.
\begin{align*}T_k=\jdt_{i_k}\cdots \jdt_{i_1}(\varnothing)\end{align*}
If we let $\lambda_{k}$ denote $\sha(T_k)\vdash k$, we get a maximal chain in $\mathcal{Y}$.
\begin{align*}\varnothing \prec \lambda_1\prec \cdots \prec \lambda_n
\end{align*} 
Thus, we have that $w=i_n\cdots i_1$ is the column growth word corresponding to the maximal chain above. This combined with the reversibility of jeu de taquin for $\SSRT$s gives us the key fact that $T_{w}=e(T)$ (essentially as evacuation undoes backward jeu de taquin slides).

From Corollary \ref{corollary: insertion tableau and reading word} we know that $T_{w}=P(\sigma^{-1})$, and by Lemma \ref{lemma: basic properties of RSK}, we know that $P(\sigma^{-1})=e(Q(\sigma))$. Combining the previous two facts, we obtain the following.
\begin{align*}
e(T)= e(Q(\sigma))
\end{align*} 
Since evacuation is an involution as is shown in the appendix, we get that $T=Q(\sigma)$. This implies that the $\SRCT$ given by $\mu_{i_n}\cdots \mu_{i_1}(\varnothing)=\tau=\rho^{-1}(T)$, equals $\rho^{-1}(Q(\sigma))$.
\end{proof}

We claim that the above theorem also supplies us with a bijection between the set of maximal chains from $\varnothing$ to $\alpha$ in $\mathcal{R}_c$ and the set of $\SRCT$s of shape $\alpha$.
\begin{proof}(of Theorem \ref{theorem: counting maximal chains in right pieri poset})
Consider a maximal chain in $\mathcal{R}_c$ as shown below.
\begin{align*}
\varnothing <_r \alpha_1 <_r\cdots <_r\alpha_n=\alpha
\end{align*}
Just as in the case of $\mathcal{Y}$, the maximal chain above gives us a reverse lattice word $w=i_n\cdots i_1$ where 
$\alpha_k= u_{i_k}(\alpha_{k-1})$ for $1\leq k\leq n$. Associate the $\SRCT$ $\mu_{i_n}\cdots \mu_{i_1}(\varnothing)$ with this maximal chain. Note that the shape of this $\SRCT$, by Theorem \ref{shapemu}, is $u_{i_n}\cdots u_{i_1}(\varnothing)$ where $\varnothing$ denotes the empty composition. But this is clearly $\alpha$. To establish that the map associating the $\SRCT$ $\mu_{i_n}\cdots \mu_{i_1}(\varnothing)$ to the reverse lattice word $i_n\cdots i_1$ obtained from the maximal chain is indeed a bijection, we will use the invertibility of jeu de taquin slides for $\SSRT$s.

Consider an $\SRCT$ $\tau$ of shape $\alpha$, and let $T=\rho(\tau)$. We will analyze the execution of the evacuation action on $T$. At every step a forward jeu de taquin slide is initiated from the lower left corner, and this slide finishes in an addable node to the shape underlying the rectified tableau. Noting down the columns of these addable nodes gives us a sequence of $n$ positive integers $i_n$, $i_{n-1}$,\ldots, $i_1$. This sequence is \textbf{uniquely determined by} $T$ (by the invertibility of jeu de taquin slides on $\SSRT$s) and the word $w=i_n\cdots i_1$ is reverse lattice. Now since backward jeu de taquin slides undo the forward jeu de taquin slides, we get that 
\begin{align*}
T=\jdt_{i_n}\cdots \jdt_{i_1}(\varnothing),
\end{align*}
and hence that
\begin{align*}
\tau=\mu_{i_n}\cdots \mu_{i_1}(\varnothing).
\end{align*}
Since $\tau$ had shape $\alpha$ to start with, we get that $u_{i_n}\cdots u_{i_1}(\varnothing)=\alpha$ as well (here again $\varnothing$ refers to the empty composition).
Now we can associate the following (unique) maximal chain from $\varnothing$ to $\alpha$ in $\mathcal{R}_c$ to $\tau$.
\begin{align*}
\varnothing <_r u_{i_1}(\varnothing) <_r u_{i_2}u_{i_1}(\varnothing) <_r \cdots <_r u_{i_n}\cdots u_{i_1}(\varnothing)=\alpha.
\end{align*}
This finishes the proof.
\end{proof}
We will give an example illustrating the ideas above.
\begin{Example}\label{example: shape given column growth word}
To compute $\tau=\mu_3\mu_4\mu_1\mu_1\mu_2\mu_3\mu_1\mu_2\mu_1 (\varnothing)$, note that the column growth word is $341123121$ and its standardization (which is a permutation in $\mathfrak{S}_{9}$) in single line notation is $\sigma=791258364$. On applying the Robinson-Schensted variant, $\sigma$ maps to the following insertion tableau (left) and recording tableau (right).
\begin{align*}
P(\sigma)=\ytableausetup{mathmode,boxsize=1.2em}
\begin{ytableau}
1\\2\\7&5&3\\9&8&6&4
\end{ytableau}
\hspace{10mm}
Q(\sigma)=\ytableausetup{mathmode,boxsize=1.2em}
\begin{ytableau}
4\\5\\8&6&2\\9&7&3&1
\end{ytableau}
\end{align*}
Then, by Theorem \ref{theorem: jdt sequences and recording tableau}, we have that $\tau=\rho^{-1}(Q(\sigma))$.
\begin{eqnarray*}
\tau=\ytableausetup{mathmode,boxsize=1.2em}
\begin{ytableau}
4\\5\\8&7&3&1\\9&6&2
\end{ytableau}
\end{eqnarray*}
It is straightforward to check that $\jdt_3\jdt_4\jdt_1\jdt_1\jdt_2\jdt_3\jdt_1\jdt_2\jdt_1 (\varnothing)$ gives the following sequence of $\SRT$s (the bullets indicating where the next slide begins).
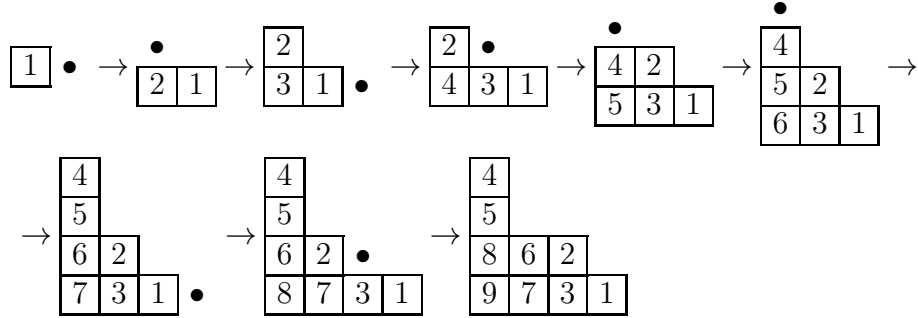
\begin{align*}
\ytableausetup{mathmode,boxsize=1.2em}
&\begin{ytableau}
1 & \none[\bullet]
\end{ytableau}
\rightarrow
\begin{ytableau}
\none[\bullet]\\
2 & 1
\end{ytableau}
\rightarrow
\begin{ytableau}
2\\
3 & 1 &\none[\bullet]
\end{ytableau}
\rightarrow
\begin{ytableau}
2 & \none[\bullet]\\
4 & 3&1
\end{ytableau}
\rightarrow
\begin{ytableau}
\none[\bullet]\\
4&2\\
5&3&1
\end{ytableau}
\rightarrow
\begin{ytableau}
\none[\bullet]\\
4\\
5&2\\
6&3&1
\end{ytableau}\rightarrow
\\ &\rightarrow
\begin{ytableau}
4\\
5\\
6&2\\
7&3&1 &\none[\bullet]
\end{ytableau}
\rightarrow
\begin{ytableau}
4\\
5\\
6&2 &\none[\bullet]\\
8&7&3&1
\end{ytableau}
\rightarrow
\ytableausetup{mathmode,boxsize=1.2em}
\begin{ytableau}
4\\5\\8&6&2\\9&7&3&1
\end{ytableau}
\end{align*}
Note that the final $\SRT$ obtained above is $\rho(\tau)$.
\end{Example}
In the next subsection, we will use Theorem \ref{theorem: jdt sequences and recording tableau} and the proof strategy for Theorem \ref{theorem: counting maximal chains in right pieri poset} to consider the effect of a sequence of slides starting from a nonempty $\SRCT$.
\subsection{ Inner shape after a sequence of slides}
Let $\tau_1$ be an $\SRCT$ of shape $\alpha\vDash n$, and let $T_1=\rho(\tau_1)$ be of shape $\lambda\vdash n$. Considering the execution of the evacuation action on $T_1$ as in the proof of Theorem \ref{theorem: counting maximal chains in right pieri poset}, we get a sequence of $n$ positive integers $i_n$, $i_{n-1}$,\ldots, $i_1$. This sequence is uniquely determined by $T_1$ and the word $w=i_n\cdots i_1$ is reverse lattice.
\begin{Remark}
Note that if we started with the final $\SRT$ in Example \ref{example: shape given column growth word}, and performed the procedure discussed prior to this remark, we will be obtain the same sequence of $\SRT$s in reverse and the columns to which the bullets belong are precisely the integers $i_n$ down to $i_1$.
\end{Remark}
It follows that $T_1=\jdt_{i_n}\cdots \jdt_{i_1}(\varnothing)$
which in turn implies that $\tau_1=\mu_{i_n}\cdots \mu_{i_1}(\varnothing)$.
Consider now the $\SRCT$ $\tau_2$ defined below.
\begin{align*}
\tau_2=\mu_{k_m}\cdots \mu_{k_1}(\tau_1)
\end{align*}
Let $\rho(\tau_2)$ be $T_2$. Clearly, we have that $T_2=\jdt_{k_m}\cdots\jdt_{k_1}(T_1)$. Denote by $u$ the word $k_m\cdots k_1$.

Let $v=u\cdot w$, where $\cdot$ denotes concatenation of words, and let $\pi=\std(v)$, $\sigma=\std(w)$ and $\omega=\std(u)$. Now, note that Theorem \ref{theorem: jdt sequences and recording tableau} implies that 
\begin{align*}
\tau_2=\rho^{-1}(Q(\pi)).
\end{align*}
Observe additionally that the integers $m+n$, $m+n-1,\ldots$, $n+1$ are the new entries added when computing $\mu_{k_m}\cdots \mu_{k_1}(\tau_1)$, and they form an $\SSRCT$ of size $m$ (with all entries distinct) inside $\tau_2$. But this $\SSRCT$ is obtained by performing our insertion algorithm on the prefix of length $m$ in $\pi$, and then applying $\rho^{-1}$ to the partial $Q$-tableau obtained at this stage. If one subtracts $n$ from all entries in this partial $Q$-tableau, one gets an $\SRT$ that is precisely the $Q$-tableau obtained after performing the insertion process on $\omega=\std(u)$ (recall that standardization does not alter the $Q$-tableau).

From this point on, we switch perspectives again, and on performing backward jeu de taquin slides on $\SSRCT$s, instead of filling the vacant box in the lower left corner with larger entries, we will fill them with bullets. The preceding discussion implies the following theorem.
\begin{Theorem}\label{lemma: predicting inner skew shape after slides}
Let $\tau_1$ be an $\SRCT$ of straight shape $\alpha$, and $\tau_2$ be the $\SRCT$ $\mu_{j_m}\cdots \mu_{j_1}(\ctau_1)$. Let $w=j_m\cdots j_1$ and $\omega=\std(w)\in\mathfrak{S}_m$. Let $\beta=u_{j_m}\cdots u_{j_1}(\alpha)$ and $\gamma = \sha(\rho^{-1}(Q(\omega)))$. Then $\tau_2$ is an $\SRCT$ of shape $\beta \cskew \gamma$.
\end{Theorem}

The following corollary is obtained from the theorem above.
\begin{Corollary}\label{corollary: slides to get row/column shape}
Let $\tau_1$ be an $\SRCT$ of straight shape $\alpha$, and $\tau_2$ be the $\SRCT$ $\mu_{j_m}\cdots \mu_{j_1}(\ctau_1)$. Let $w=j_m\cdots j_1$ and $\omega=\std(w)\in\mathfrak{S}_m$. Let $\beta=u_{j_m}\cdots u_{j_1}(\alpha)$. Then $\tau_2$ is an $\SRCT$ of shape $\beta \cskew (n)$ if and only if $j_1<\cdots < j_m$, and it is an $\SRCT$ of shape $\beta\cskew (1^n)$ if and only if $j_1\geq \cdots \geq  j_m$.
\end{Corollary}

$\SRT$ versions of the two results above also exist and they are also implied in the discussion earlier. We will recast the corollary above for $\SRT$s as another corollary.
\begin{Corollary}\label{lemma: SRT slides to get row/column shape}
 Let $T_1$ be an $\SRT$ of straight shape $\lambda$, and $T_2$ be the $\SRT$ $\jdt_{j_m}\cdots \jdt_{j_1}(T_1)$. Let $w=j_m\cdots j_1$ and $\omega=\std(w)\in\mathfrak{S}_m$. Let $\beta=u_{j_m}\cdots u_{j_1}(\alpha)$ and $\delta=\widetilde{\beta}$. Then $T_2$ is an $\SRT$ of shape $\delta /(n)$ if and only if $j_1<\cdots < j_m$, and it is an $\SRT$ of shape $\delta / (1^n)$ if and only if $j_1\geq \cdots \geq  j_m$.
\end{Corollary}

Suppose now that $T$ is an $\SRT$ of shape $\delta / (n)$. Then the invertibility of the jeu de taquin slides for $\SRT$s along with the corollary above implies that $T$ has been obtained from an $\SRT$ $T'$ by doing a sequence of backward jeu de taquin slides as follows
\begin{align*}
T=\jdt_{i_n}\cdots\jdt_{i_1}(T'),
\end{align*}
where $i_n>\cdots >i_1$. Similarly, if $T$ is an $\SRT$ of shape $\delta / (1^n)$, then it has been obtained from an $\SRT$ $T'$ by doing a sequence of backward jeu de taquin slides as follows
\begin{align*}
T=\jdt_{i_n}\cdots\jdt_{i_1}(T'),
\end{align*}
where $i_n\leq \cdots \leq i_1$. Note now that the maps $\rho_{(n)}\!: \SSRT(/ (n)) \to \SSRCT(\cskew (n))$, and $\rho_{(1^n)}: \SSRT(/ (1^n)) \to \SSRCT(\cskew (1^n))$ are bijections. Using this and Theorem \ref{backwardjdtslideonssrct}, it is clear that we can apply the above arguments to $\SRCT$s. This insight is all we need to prove our right Pieri rule for the noncommutative Schur functions stated in Theorem \ref{theorem: right pieri rules}.

\section{A right Pieri rule for noncommutative Schur functions}\label{section: right pieri rule}
To recover a Pieri rule from Theorem \ref{theorem: noncommutative LR rule} that allows us to compute the products $\ncsa\cdot \ncs_{(n)}$ and $\ncsa\cdot\ncs_{(1^n)}$ we need to find all $\SRCT$s of shape $\gamma\cskew (n)$ (respectively $\gamma\cskew (1^n)$) that rectify to the canonical tableau of shape $\alpha$. In light of Corollary \ref{corollary: slides to get row/column shape} we have the following lemma.

\begin{Lemma}\label{lemma: slides in increasing order right pieri}
Let $\tau=\mu_{i_n}\cdots\mu_{i_1}(\tau_{\alpha})$, where $i_n>\cdots >i_1$. Then $\tau$ rectifies to $\tau_{\alpha}$.
\end{Lemma}
\begin{proof}
Let $T_{\alpha}=\rho(\tau_{\alpha})$ and $T=\rho_{(n)}(\tau)$. Then, by Theorem \ref{backwardjdtslideonssrct}, we have that $T=\jdt_{i_n}\cdots \jdt_{i_1}(T_{\alpha})$, which in turn implies that $\rect(T)=\rect(T_{\alpha})$ (as discussed before Algorithm \ref{algorithm:evacuation}). Therefore, we get that $\tau$ rectifies to $\tau_{\alpha}$.
\end{proof}

\begin{Remark}\label{remark: important remark}
The arguments after Corollary \ref{lemma: SRT slides to get row/column shape} imply that if there is an $\SRCT$ of shape $\gamma\cskew (n)$ that rectifies to $\tau_{\alpha}$, then it has to be obtained from $\tau_{\alpha}$ by doing a sequence of backward jeu de taquin slides from columns $i_1$, $i_{2},\cdots, i_n$ in that order, where $i_n>\cdots >i_1$.

Similarly, if there is an $\SRCT$ of shape $\gamma\cskew (1^n)$ that rectifies to $\tau_{\alpha}$, then it has to be obtained from $\tau_{\alpha}$ by doing a sequence of backward jeu de taquin slides from columns $i_1$, $i_{2},\cdots, i_n$ in that order, where $i_n\leq \cdots \leq i_1$.
\end{Remark}
We can now prove the right Pieri rule, stated in Theorem \ref{theorem: right pieri rules}.
\begin{Theorem}(Right Pieri rule)\label{theorem: right pieri rules actual proof}
Let $\alpha$ be a composition and $n$ a positive integer. Then
\begin{eqnarray*}
\ncsa\cdot \ncs_{(n)}=\displaystyle\sum_{\substack{\beta\vDash |\alpha|+n, u_{i_n}\cdots u_{i_1}(\alpha)=\beta\\i_n>\cdots >i_1}} \ncsb ,\\
\ncsa\cdot \ncs_{(1^n)}=\displaystyle\sum_{\substack{\beta\vDash |\alpha|+n, u_{i_n}\cdots u_{i_1}(\alpha)=\beta\\i_n\leq \cdots \leq i_1}} \ncsb .
\end{eqnarray*}
\end{Theorem}
\begin{proof}
We will only give the proof for $\ncsa\cdot \ncs_{(n)}$ as the proof of $\ncsa\cdot \ncs_{(1^n)}$ is similar. Notice first that if we have two distinct sequences $i_n>\cdots >i_1$ and $j_n>\cdots >j_1$ such that 
\begin{align*}
\tau_1=\mu_{i_n}\cdots \mu_{i_1}(\tau_{\alpha}),\\
\tau_2=\mu_{j_n}\cdots \mu_{j_1}(\tau_{\alpha}),
\end{align*}
then $\tau_1$ and $\tau_2$ are distinct $\SRCT$s. This follows from the reversibility of backward jeu de taquin slides for $\SRT$s, once we consider the images of $\tau_1$ and $\tau_2$ under the generalized $\rho$ map of Subsection \ref{neededlater}. By Corollary \ref{corollary: slides to get row/column shape}, we get that 
\begin{align*}
\sha(\tau_1)=\gamma_1\cskew (n) \text{ where } \gamma_1=u_{i_n}\cdots u_{i_1}(\alpha),\\
\sha(\tau_2)=\gamma_2\cskew (n)\text{ where } \gamma_2=u_{j_n}\cdots u_{j_1}(\alpha).
\end{align*}
We claim that $\gamma_1\neq \gamma_2$. But in view of Remark \ref{remark: up operator on partition} and given that the sequences $i_n>\cdots >i_1$ and $j_n>\cdots >j_1$ are distinct, this immediately follows. Using Remark \ref{remark: important remark}, we get that all summands $\ncsg$ in $\ncsa\cdot\ncs_{(n)}$ are obtained by picking a sequence $i_n>\cdots >i_1$ such that $\mu_{i_n}\cdots \mu_{i_1}(\tau_{\alpha})$ is an $\SRCT$, and then letting $\gamma$ be the outer shape of this $\SRCT$. This gives us a multiplicity free expansion in the noncommutative Schur basis for the product $\ncsa\cdot\ncs_{(n)}$ as stated in the claim.
\end{proof}

We conclude by proving a generalization of Theorem \ref{theorem: counting maximal chains in right pieri poset}. Instead of counting maximal chains starting from $\varnothing$ and ending at some composition $\beta$, we consider maximal chains starting from a composition $\alpha$ and ending at a composition $\beta$ where $\alpha < _{r} \beta$.
\begin{Theorem}\label{theorem: generalized counting maximal chains in right pieri poset}
Let $f_{\alpha,\beta}$ denote the number of maximal chains from $\alpha$ to $\beta$ in $\mathcal{R}_c$. Then
\begin{eqnarray*}
f_{\alpha,\beta}=\displaystyle\sum_{\gamma \vDash |\beta|-|\alpha|}C_{\alpha\gamma}^{\beta}f_{\varnothing,\gamma}.
\end{eqnarray*}
\end{Theorem}
\begin{proof}
Let $|\alpha|=m$ and $|\beta|=n+m$.
If $\alpha<_{r}\beta$, then there is at least one sequence of positive integers $i_n,\cdots,i_1$ such that $u_{i_n}\cdots u_{i_1}(\alpha)=\beta$. Let $w=i_n\cdots i_1$ consider the $\SRCT$ $\tau=\mu_{i_n}\cdots\mu_{i_1}(\tau_{\alpha})$. Then $\sha(\tau)=\beta\cskew \gamma$ where $\gamma=\sha (\rho^{-1}(Q(\std(w))))$, by Theorem \ref{lemma: predicting inner skew shape after slides}. Note also that $\tau$ rectifies to $\tau_{\alpha}$, by an argument similar to the proof of Lemma \ref{lemma: slides in increasing order right pieri}.

As stated in Theorem \ref{theorem: noncommutative LR rule}, we have that
\begin{eqnarray*}
C_{\alpha\gamma}^{\beta}= \text{ number of $\SRCT$s of shape $\beta\cskew \gamma$ that rectify to $\tau_{\alpha}$}.
\end{eqnarray*}
Now consider any skew $\SRCT$ $\tau$ of shape $\beta\cskew \gamma$ that rectifies to $\tau_{\alpha}$. This can be completed to an $\SRCT$ $\tau_{aug}$ of straight shape $\beta$ where the inner shape corresponding to $\gamma$ has been filled according to an $\SSRCT$ of shape $\gamma$ with distinct entries from the set $B=\{n+m,n+m-1,\ldots,m+1\}$. Now, notice that every choice of the filling of shape $\gamma$ gives a unique maximal chain from $\alpha$ to $\beta$. To see this, perform a forward jeu de taquin slide starting from the lower left corner of $\rho(\tau_{aug})$. As long as the resulting tableau has elements from $B$, repeat the previous step. Notice that the procedure above stops when we reach $\rho(\tau_{\alpha})$. Reversing these slides starting from $\rho ({\tau_{\alpha}})$, and using the generalized $\rho$ map, defined in Subsection \ref{neededlater}, at every step gives us the unique maximal chain from $\alpha$ to $\beta$.

Using the fact that the number of $\SRCT$s of shape $\gamma$ is $f_{\varnothing,\gamma}$, as established in Theorem \ref{theorem: counting maximal chains in right pieri poset}, the claim follows.
\end{proof}
\begin{Remark}
It is worth emphasizing that $f_{\alpha,\beta}$ is not the number of $\SRCT$ of shape $\beta\cskew \alpha$ in general, although it is true in the case where $\alpha$ is $\varnothing$ by Theorem \ref{theorem: counting maximal chains in right pieri poset}. The above theorem can be considered as a noncommutative generalization of extracting coefficients in the classical identity $s_{\lambda/\mu}=\displaystyle\sum_{\nu\vdash |\lambda|-|\mu|}c_{\mu\nu}^{\lambda}s_{\nu}$.
\end{Remark}

\appendix
\section*{Appendix}
We will outline the relation between our variant of the Robinson-Schensted algorithm and the classical Robinson-Schensted algorithm. The latter assigns to a permutation $\sigma\in \mathfrak{S}_n$ a pair of standard Young tableaux of the same shape and size equalling $n$. We begin by defining standard Young tableaux.

Given a partition $\lambda$, a \emph{standard Young tableau (SYT)} $T$ of \emph{shape} $\lambda$ is a filling of the boxes of $\lambda$ with distinct positive integers from 1 to $\vert \lambda \vert$, satisfying the condition that the entries in $T$ are strictly increasing along each row read from left to right and strictly increasing along each column read from bottom to top.

The classical Robinson-Schensted algorithm and evacuation action for SYTs (which we will denote by $\evac$) are discussed in detail in \cite{sagan, stanley-ec2} and we will not present them here. To link the classical algorithms for SYTs with the variant for $\SRT$s outlined in Subsection \ref{subsection: vrsk}, we will need the notion of the complement of a permutation, and the complement of an SYT/$\SRT$.

Given a permutation $\sigma = \sigma(1)\cdots \sigma(n)$ in single line notation, the \textit{complement} of $\sigma$, denoted by $\sigma^{c}$ is obtained by replacing $\sigma(i)$ by $n+1-\sigma(i)$ for $1\leq i\leq n$. Similarly, given an SYT ($\SRT$) $T$ of shape $\lambda\vdash n$, the \textit{complement} of $T$ is the $\SRT$ (SYT), denoted by $T^{c}$, obtained by replacing an entry $k$ in $T$ by $n+1-k$. Finally, define $T^t$ to be the transpose of any tableau $T$.  

Observe that the complement and transpose operators on SYTs/$\SRT$s commute. Furthermore, both the $\evac$ operator on $SYTs$ and the $e$ operator on $\SRT$s commute with the transpose operator. Finally, note that
\begin{align*}
e(T)=(\evac(T^c))^c \text{ if $T$ is an $\SRT$,}\\
\evac(T)=(e(T^c))^c \text{ if $T$ is an SYT.}
\end{align*}
Since $\evac$ is an involution by \cite[Proposition A.1.2.9]{stanley-ec2}, we can see that $e$ is also an involution.

Now we are ready to make explicit the relation between the classical algorithm for SYTs and the one for $\SRT$s. Let $\pi\in \mathfrak{S}_n$. We will denote the insertion and recording tableaux obtained by using the classical Robinson-Schensted correspondence by $P(\pi)$ and $Q(\pi)$ respectively, and denote those obtained by using our variant by $P_{v}(\pi)$ and $Q_{v}(\pi)$ respectively. 

Observe that 
\begin{align*}
(P_{v}(\pi), Q_{v}(\pi)) = ((P(\pi^{c}))^{c}, (Q(\pi^{c}))^{c}).
\end{align*}
Further using \cite[Theorem A.1.2.10 and Corollary A.1.2.11]{stanley-ec2} and the fact that $\evac$ is an involution, we get that 
\begin{align*}
(P(\pi^{c}), Q(\pi^{c}) )=(\evac(P(\pi)^t),Q(\pi)^t).
\end{align*}
The two equalities above together imply that
\begin{align*}
(P_{v}(\pi), Q_{v}(\pi))=((\evac(P(\pi)^t))^c,(Q(\pi)^t)^c).
\end{align*}

Using \cite[Corollary A.1.2.11]{stanley-ec2} and the fact that $P(\pi^{-1})=Q(\pi)$ we get that
\begin{align*}
P_{v}(\pi^{-1})&=(\evac(P(\pi^{-1})^t))^c\\&=(\evac(Q(\pi)^t))^c.
\end{align*}
Since $(Q_{v}(\pi))^c=Q(\pi)^t$, we get that
\begin{align*}
P_v(\pi^{-1})&=(\evac(Q_v(\pi)^c))^c\\&=e(Q_v(\pi)).
\end{align*}

Again, using \cite[Corollary A.1.2.11]{stanley-ec2} and the fact that $Q(\pi^{-1})=P(\pi)$ we get that
\begin{align*}
Q_v(\pi^{-1}) &= (Q(\pi^{-1})^t)^c\\&=(P(\pi)^t)^c.
\end{align*}
Now, since $P(\pi)^t=\evac(P_v(\pi)^c)$, we get that
\begin{align*}
Q_v(\pi^{-1}) &= (\evac(P_v(\pi)^c))^c\\&=e(P_v(\pi)).
\end{align*}
This establishes Lemma \ref{lemma: basic properties of RSK}.

\end{document}